\documentclass[12pt,a4paper]{amsart}

\usepackage{fancyhdr}
\usepackage{appendix}
\usepackage{amssymb,amscd,amsxtra,calc}
\usepackage{enumitem}
\usepackage{mathrsfs}
\usepackage{mathtools}
\usepackage{multirow}
\usepackage[all]{xy}
\usepackage[T1]{fontenc}
\usepackage{longtable}
\usepackage[colorlinks,linkcolor=blue,anchorcolor=blue,citecolor=green]{hyperref}

\usepackage[capitalise]{cleveref}

\theoremstyle{plain}
    \newtheorem{thm}{Theorem}[section]
    
    \newtheorem{claim}[thm]{Claim}
    \newtheorem{conjecture}[thm]{Conjecture}
    \newtheorem{corollary}[thm]{Corollary}
    \newtheorem{example}[thm]{Example}
    \newtheorem{lemma}[thm]{Lemma}
    \newtheorem{proposition}[thm]{Proposition}
    
    \newtheorem{theorem}[thm]{Theorem}
    
    \newtheorem*{Thm_k>0_str}{Theorem A}

\theoremstyle{definition}
    \newtheorem{definition}[thm]{Definition}
    
    \newtheorem*{notation*}{Notation and Terminology}
    \newtheorem{remark}[thm]{Remark}
    \newtheorem*{ack}{Acknowledgments}
\theoremstyle{remark}

    \newtheorem{setup}[thm]{}

\newcommand{\RomanNumeralCaps}[1]{\MakeUppercase{\romannumeral #1}}
\newcommand{\kk}{\textbf{k}}

\newcommand{\A}{\mathbb{A}}
\newcommand{\C}{\mathbb{C}}

\newcommand{\G}{\mathbb{G}}
\newcommand{\PP}{\mathbb{P}}
\newcommand{\Q}{\mathbb{Q}}
\newcommand{\R}{\mathbb{R}}
\newcommand{\Z}{\mathbb{Z}}
\newcommand{\OO}{\mathcal{O}}

\newcommand{\et}{\mathrm{\acute{e}t}}
\newcommand{\id}{\operatorname{id}}
\newcommand{\okappa}{\operatorname{\overline{\kappa}}}
\newcommand{\ord}{\operatorname{ord}}
\newcommand{\red}{\mathrm{red}}

\newcommand{\Aut}{\operatorname{Aut}}
\newcommand{\Exc}{\operatorname{Exc}}
\newcommand{\Fix}{\operatorname{Fix}}
\newcommand{\Gal}{\operatorname{Gal}}
\newcommand{\Ker}{\operatorname{Ker}}
\newcommand{\N}{\operatorname{N}}
\newcommand{\Nef}{\operatorname{Nef}}
\newcommand{\NS}{\operatorname{NS}}
\newcommand{\Pic}{\operatorname{Pic}}
\newcommand{\Sing}{\operatorname{Sing}}
\newcommand{\Spec}{\operatorname{Spec}}
\newcommand{\Supp}{\operatorname{Supp}}

\begin{document}

\title[Endomorphisms of varieties]
{Endomorphisms of quasi-projective varieties: towards Zariski dense orbit and Kawaguchi-Silverman conjectures}

\author{Jia Jia}
\address{Department of Mathematics, National University of Singapore, Singapore 119076, Republic of Singapore}
\curraddr{Yau Mathematical Sciences Center, Tsinghua University, Haidian District, Beijing 100084, China}
\email{jia\_jia@nus.edu.sg,mathjiajia@tsinghua.edu.cn}

\author{Takahiro Shibata}
\address{Department of Mathematics, National University of Singapore, Singapore 119076, Republic of Singapore}
\curraddr{Faculty of Human Development and Culture, Fukushima University, 1 Kanayagawa, Fukushima, 960-1248, Japan}
\email{t-shibata@educ.fukushima-u.ac.jp}

\author{Junyi Xie}
\address{BICMR, Peking University, Haidian District, Beijing 100871, China}
\email{xiejunyi@bicmr.pku.edu.cn}

\author{De-Qi Zhang}
\address{Department of Mathematics, National University of Singapore, Singapore 119076, Republic of Singapore}
\email{matzdq@nus.edu.sg}

\begin{abstract}
	Let $X$ be a quasi-projective variety and $f\colon X\to X$ a finite surjective endomorphism.
	We consider Zariski Dense Orbit Conjecture (ZDO),
	and Adelic Zariski Dense Orbit Conjecture (AZO).
	We consider also Kawaguchi-Silverman Conjecture (KSC)
	asserting that the (first) dynamical degree $d_1(f)$ of $f$
	equals the arithmetic degree $\alpha_f(P)$ at a point $P$ having Zariski dense $f$-forward orbit.
	Assuming $X$ is a smooth affine surface,
	such that the log Kodaira dimension $\overline{\kappa}(X)$ is non-negative
	(resp. the \'etale fundamental group $\pi_1^{\text{\'et}}(X)$ is infinite),
	we confirm AZO, (hence) ZDO, and KSC (when $\deg(f)\geq 2$) (resp. AZO and hence ZDO).
	We also prove ZDO (resp. AZO and hence ZDO)
	for every surjective endomorphism on any projective variety
	with ``larger'' first dynamical degree
	(resp. every dominant endomorphism of any semiabelian variety).
\end{abstract}

\maketitle
\tableofcontents

\section{Introduction}

We work over an algebraically closed field $\kk$ of characteristic $0$ unless otherwise stated.

The motivation of the paper is the following
\emph{Zariski Dense Orbit Conjecture} (ZDO),
\emph{Adelic Zariski Dense Orbit Conjecture} (AZO)
and \emph{Kawaguchi-Silverman Conjecture} (KSC).

\begin{conjecture}[Zariski Dense Orbit Conjecture=ZDO]
	\label{Conj:ZDO}
	Let $X$ be a variety over $\kk$ and $f\colon X\dashrightarrow X$ a dominant rational map.
	If the $f^*$-invariant function field $\kk(X)^f=\kk$,
	then there exists some $x\in X(\kk)$
	whose $f$-orbit $O_f(x)\coloneqq\{f^n(x)\mid n\geq 0\}$ is well-defined,
	i.e., $f$ is defined at $f^n(x)$ for any $n\geq 0$,
	and Zariski dense in $X$.
\end{conjecture}

When the transcendence degree of $\kk$ over $\Q$ is finite, in \cite[Section 3]{Xie19},
the third author has introduced the adelic topology on $X(\kk)$
and proposed the Adelic Zariski Dense Orbit Conjecture (AZO).

\begin{setup}[Adelic Topology]\label{setup:AdelicTopology}
	The adelic topology has the following basic properties (cf.~\cite[Proposition~3.18]{Xie19}).
	\begin{enumerate}[leftmargin=2em]
		\item It is stronger than the Zariski topology.
		\item It is ${\mathsf{T}}_1$,
		      i.e., for every distinct points $x, y \in X(\kk)$
		      there are adelic open subsets $U, V$ of $X(\kk)$ such that
		      $x \in U, y \notin U$ and $y\in V, x \notin V$.
		\item Morphisms between algebraic varieties over $\kk$ are continuous for the adelic topology.
		\item Flat morphisms are open with respect to the adelic topology.
		\item \label{ppty-of-adelic-topo:irreducible}
		      The irreducible components of $X(\kk)$ in the Zariski topology
		      are the irreducible components of $X(\kk)$ in the adelic topology.
		\item Let $K$ be any subfield of $\kk$ which is finitely generated over $\mathbb{Q}$
		      and such that $X$ is defined over $K$ and $\overline{K} = \kk$.
		      Then the action
		      \[
			      \Gal(\kk/K)\times X(\kk)\to X(\kk)
		      \]
		      is continuous with respect to the adelic topology.
	\end{enumerate}
\end{setup}

\begin{remark}\label{rem:intersection_of_finite_adelic_open} $ $
	\begin{enumerate}[leftmargin=2em]
		\item When $X$ is irreducible,
		      the property (\ref{ppty-of-adelic-topo:irreducible}) above implies that
		      the intersection of finitely many nonempty adelic open subsets of $X(\kk)$ is nonempty.
		      So, if $\dim X\geq 1$, the adelic topology is not Hausdorff.
		      In general, the adelic topology is strictly stronger than the Zariski topology.

		\item An impotent example of adelic open subsets is as follows:
		      Let $L$ be a subfield of $\kk$
		      such that its algebraic closure ${\overline{L}}$ is equal to $\kk$,
		      $L$ is finitely generated over $\Q$,
		      and $X$ is defined over $L$,
		      i.e., $X=X_L\otimes_L \kk$ for some variety $X_L$ over $L$.
		      Fix any embedding $\tau\colon L\hookrightarrow \C_p$ (resp. $\C$).
		      Then, given any open subset $U$ of $X_L(\C_p)$ for the $p$-adic (resp.~Euclidean) topology,
		      the union $X_L(\tau, U)\coloneqq \cup_\iota \Phi_\iota^{-1}(U)$
		      for all embeddings $\iota\colon \kk \to \C_p$
		      extending $\tau$ is,
		      by definition, an open subset of $X(\kk)$ in the adelic topology.
		      Moreover $X_L(\tau, U)$ is empty if and only if $U=\emptyset$.
	\end{enumerate}
\end{remark}

\begin{conjecture}[Adelic Zariski Dense Orbit Conjecture=AZO]\label{Conj:AZO}
	Assume the hypothesis that the transcendence degree of $\kk$ over $\Q$ is finite.
	Let $X$ be a variety over $\kk$ and $f\colon X\dashrightarrow X$ a dominant rational map.
	If the $f^*$-invariant function field $\kk(X)^f=\kk$,
	then there exists a nonempty adelic open subset $A\subseteq X(\kk)$
	such that for every point $x\in A$ the $f$-orbit $O_f(x)$ is well-defined and Zariski dense in $X$.
\end{conjecture}

\begin{conjecture}[Kawaguchi-Silverman Conjecture=KSC\@; cf.~\cite{KS16}]\label{Conj:KSC}
	Let $X$ be a quasi-projective variety over $\overline{\Q}$
	and $f\colon X\dashrightarrow X$ a dominant rational map.
	Take a point $x\in X(\overline{\Q})$.
	If the $f$-orbit $O_f(x)$ is well-defined and Zariski dense in $X$,
	then the limit $\alpha_f(x)$ (called the arithmetic degree) as defined in \ref{subsec_arithdeg},
	converges and equals $d_1(f)$, the first dynamical degree of $f$ (cf.~\ref{dyn_deg}).
\end{conjecture}

\begin{remark}\label{rem:AZO_to_ZDO}
	AZO~\ref{Conj:AZO} implies ZDO~\ref{Conj:ZDO}.
	Indeed, even the hypothesis on $\kk$ in AZO~\ref{Conj:AZO} does not cause any problem.
	To be precise,
	for every pair $(X,f)$ over any algebraically closed field $\kk$ of characteristic zero,
	there exists an algebraically closed subfield $K$ of $\kk$
	whose transcendence degree over $\Q$ is finite
	and such that $(X,f)$ is defined over $K$,
	i.e., there exists a pair $(X_K,f_K)$ such that $(X,f)$ is its base change by $\kk$.
	By \cite[Corollary~3.31]{Xie19},
	if AZO~\ref{Conj:AZO} holds for $(X_K,f_K)$, then ZDO~\ref{Conj:ZDO} holds for $(X,f)$.
\end{remark}

\textit{In this paper, when we discuss AZO~\ref{Conj:AZO},
	we always assume that the transcendence degree of $\kk$ over $\Q$ is finite;
	when we discuss KSC~\ref{Conj:KSC}, we always assume that $\kk=\overline{\Q}$.}

\medskip

Theorems~\ref{Thm_semiAb}, \ref{Thm_k>0_str}, \ref{Thm_k>0_conj},
\ref{Thm_inf_pi1}, \ref{Thm_tri_poly} and \ref{Thm_big_d1} are our main results.

We first deal with endomorphisms of semiabelian varieties:
in \cref{Thm_semiAb} below,
we prove AZO~\ref{Conj:AZO} and hence ZDO~\ref{Conj:ZDO}, while KSC~\ref{Conj:KSC} is known.

\begin{thm}\label{Thm_semiAb}
	Assume that the transcendence degree of $\kk$ over $\Q$ is finite.
	Let $G$ be a semiabelian variety over $\kk$ (cf.~\ref{NT}).
	Let $f\colon G\to G$ be a dominant endomorphism.
	Then AZO~\ref{Conj:AZO} and hence ZDO~\ref{Conj:ZDO} hold for $(G,f)$.
\end{thm}

\begin{remark}\label{r:AZO_KSC_surf}
	\leavevmode
	\begin{enumerate}[leftmargin=2em]
		\item \Cref{Thm_semiAb} generalises \cite[Theorem~1.14]{Xie19}
		      and \cite[Theorem 1.1]{GS19} from abelian varieties to semiabelian varieties.
		\item AZO~\ref{Conj:AZO} and hence ZDO~\ref{Conj:ZDO} are known
		      for surjective endomorphisms of projective surfaces (cf.~\cite{Xie19}, \cite{JXZ20}).
		\item KSC~\ref{Conj:KSC} is known
		      for surjective endomorphisms of projective surfaces (cf.~\cite{MSS18}, \cite{MZ19})
		      and for surjective endomorphisms of semiabelian varieties (cf.~\cite{MS20}).
	\end{enumerate}
\end{remark}

The next aim of this paper is to extend both AZO~\ref{Conj:AZO} and KSC~\ref{Conj:KSC} to affine surfaces.
Recall the \emph{log Kodaira dimension} $\okappa(X)$ of a variety $X$
takes value in $\{-\infty,0,\dots,\dim X\}$ (cf.~\cite[\S 11]{Iit82}).
\Cref{Thm_k>0_str} below gives the structures of endomorphisms of smooth affine surfaces;
note that when $\okappa(X)=2$, our $f$ is an automorphism of finite order (cf.~\cref{l:finord}).
For the proof of \cref{Thm_k>0_str}, the main steps are given in \cite{GZ08}.
We simplify or supplement more details in the present paper.

\begin{theorem}\label{Thm_k>0_str}
	Let $X$ be a smooth affine surface
	and $f\colon X\to X$ a finite surjective morphism of degree $\geq 2$.
	Then we have:
	\begin{enumerate}[leftmargin=2em]
		\item Suppose $\okappa(X)=0$.
		      Then $X$ is a \textit{Q}-algebraic torus (cf.~\ref{NT}).
		      Precisely, there is a finite {\'e}tale cover $\tau\colon T\to X$
		      from an algebraic torus $T\cong\G_m^2$;
		      further, $\tau$ can be chosen
		      such that $f$ lifts to a surjective morphism $f_T\colon T\to T$.

		\item Suppose $\okappa(X)=1$.
		      Then there is a finite {\'e}tale cover $X''\to X$
		      such that $X''=\G_m\times B''$ where $B''$ is a smooth affine curve
		      with $\okappa(B'')=1$,
		      and $f$ (after iteration) lifts to an endomorphism $f''$ on $X''$
		      such that $\pi''\circ f''=\pi''$,
		      where the natural projection $\pi''\colon X''\to B''$
		      is the lifting of an Iitaka fibration of $X$.
	\end{enumerate}
\end{theorem}

As a consequence of the above structural theorem,
we confirm both AZO~\ref{Conj:AZO} and KSC~\ref{Conj:KSC}
for smooth affine surfaces of non-negative log Kodaira dimension.

\begin{theorem}\label{Thm_k>0_conj}
	Let $X$ be a smooth affine surface and $f\colon X\to X$ a finite surjective endomorphism.
	Suppose $\okappa(X)\geq 0$.
	Then:
	\begin{enumerate}[leftmargin=2em]
		\item AZO~\ref{Conj:AZO} and hence ZDO~\ref{Conj:ZDO} hold for $(X,f)$.
		\item If $\kk=\overline{\Q}$ and $\deg(f)\geq 2$, then KSC~\ref{Conj:KSC} holds for $(X,f)$.
	\end{enumerate}
\end{theorem}

We may replace the assumption on the log Kodaira dimension in \cref{Thm_k>0_conj}
by the assumption on the fundamental group of $X$.
Precisely, we have:

\begin{theorem}\label{Thm_inf_pi1}
	Let $X$ be a smooth affine surface and $f\colon X\to X$ a finite surjective endomorphism.
	Suppose the {\'e}tale fundamental group $\pi_1^{\et}(X)$ is infinite.
	Then we have:
	\begin{enumerate}[leftmargin=2em]
		\item AZO~\ref{Conj:AZO} and hence ZDO~\ref{Conj:ZDO} hold for $(X,f)$.
		\item Suppose $\kk=\overline{\Q}$, $\deg(f)\geq 2$ and $X\not\cong\A^1\times\G_m$.
		      Then KSC~\ref{Conj:KSC} holds for $(X,f)$.
		\item Suppose that $\overline{\kappa}(X)=-\infty$.
		      Then $f$ descends along an $\A^1$-fibration of $X$,
		      hence KSC~\ref{Conj:KSC} holds for $(X,f)$ when $\deg(f)=1$ and $\kk=\overline{\Q}$.
	\end{enumerate}
\end{theorem}

\begin{remark}
	On \cref{Thm_inf_pi1}, we have:
	\begin{enumerate}[leftmargin=2em]
		\item In the proof of \cref{Thm_inf_pi1},
		      we actually prove a stronger statement as follows (cf.~\cref{lem:pi_1-inv}).
		      Assume that the transcendence degree of $\kk$ over $\Q$ is finite.
		      Fix an embedding $\tau\colon\kk\hookrightarrow\C$
		      and denote by $X_{\C}$ the base change of $X$ by $\C$ via $\tau$.
		      If the fundamental group $\pi_1(X_{\C})$ is infinite,
		      then AZO~\ref{Conj:AZO} holds for $(X,f)$.
		\item The case where $\pi_1^{\et}(X)$ is finite seems harder,
		      but we remark that ZDO~\ref{Conj:ZDO} holds for surjective endomorphisms of $\A^2$,
		      where $\pi_1^{\et}(\A^2)=(1)$ (cf.~\cite{Xie17}).
		\item At the moment, we are not able to prove KSC~\ref{Conj:KSC} when $X\cong\A^1\times\G_m$.
	\end{enumerate}
\end{remark}

We return to higher-dimensional ZDO~\ref{Conj:ZDO} and KSC~\ref{Conj:KSC}.
In \cref{dyn-arith},
we first prove some basic properties of the arithmetic degree
for dominant rational self-maps $f\colon X\dashrightarrow X$ on singular varieties $X$.
In particular, \cref{proupboundarth} generalises the upper bound of the arithmetic degree
by the dynamical degree from smooth varieties to singular varieties
(cf.~\cite[Theorem 1.4]{Mat20}, also \cite{KS16} when $f$ is a morphism and $X$ is smooth projective).

In \cref{Thm_tri_poly} below,
we prove AZO~\ref{Conj:AZO} for a class of upper triangular rational self-maps.
It generalises the result in \cite[Theorem~7.16]{MS14} for ZDO~\ref{Conj:ZDO}
for split polynomial endomorphisms.
A surjective endomorphism of $X=\A^1\times\G_m$ is a special case of it
(cf.~\cref{ThmC}) and is used in the proof of \cref{Thm_inf_pi1}.

\begin{theorem}\label{Thm_tri_poly}
	Let $f\colon\A^m\dashrightarrow\A^m$ be a dominant rational self-map taking the form
	\[
		(x_1,\dots,x_m)\mapsto (f_1(x_1),f_2(x_1,x_2),\dots,f_m(x_1,\dots,x_m))
	\]
	where $f_1\in\kk(x_1),f_2\in\kk(x_1)[x_2],\dots,f_m\in\kk(x_1,\dots,x_{m-1})[x_m]$.

	Then AZO~\ref{Conj:AZO} and hence ZDO~\ref{Conj:ZDO} hold for $(\A^m,f)$.
\end{theorem}

In the last part of the paper, we introduce a method to study ZDO~\ref{Conj:ZDO} via arithmetic degree.
Using this method, we obtain some sufficient conditions for ZDO~\ref{Conj:ZDO}
for endomorphisms of higher dimensional projective varieties.
Below is a sample result saying that ZDO~\ref{Conj:ZDO} holds true
for those $f$ with larger first dynamical degree $d_1(f)$.

\begin{theorem}\label{Thm_big_d1}
	Let $f\colon X\to X$ be a dominant endomorphism of a projective variety over $\overline{\Q}$.
	Then ZDO~\ref{Conj:ZDO} holds if one of the following conditions is satisfied.

	\begin{enumerate}[leftmargin=2em]
		\item $\dim X=2$, $d_1(f)>1$ and $d_1(f)\geq d_2(f)$; or
		\item $\dim X=3$, and $d_1(f)>d_3(f)=1$; or
		\item $X$ is smooth of dimension $d\geq 2$, and $d_1(f)>\max_{i=2}^d \{d_{i}(f)\}$.
	\end{enumerate}
\end{theorem}

\begin{ack}
	The first, second and fourth authors are supported by the President's scholarship,
	a Research Fellowship of NUS and ARF of NUS: A-8000020-00-00 and A-8002487-00-00;
	the third author is partially supported by the project ``Fatou'' ANR-17-CE40-0002-01.
	The authors would like to thank O.~Fujino, R.~V.~Gurjar, H.~Y.~Lin and Y.~Matsuzawa for very valuable discussions.
	The authors would also like
	to thank the referee for the very careful reading
	and suggestions to improve the paper.
\end{ack}

\section{General preliminary results}

\begin{setup}\textbf{Notation and Terminology}\label{NT}

	By a \emph{variety} we mean an algebraic variety,
	i.e., an integral separated scheme of finite type over the field $\kk$.
	On a smooth variety $V$,
	a reduced divisor $D$ is of \emph{simple normal crossing} (SNC for short)
	if every irreducible component of $D$ is smooth and locally $D=\{x_1\cdots x_s=0\}$
	at every point of $D$ with local coordinates $x_1,\dots,x_n$ of $V$.

	Let $X$ be a smooth quasi-projective variety.
	A pair $(V,D)$ is called a \emph{compactification} of $X$
	if $V$ is a projective variety containing $X$ as an open subset and $D=V\setminus X$.
	A compactification $(V,D)$ of $X$ is a \emph{log smooth compactification}
	if $V$ is smooth and $D$ is an SNC divisor.

	Given a log smooth compactification $(V,D)$ of $X$,
	the \emph{log Kodaira dimension} $\overline{\kappa}(X)$ is defined as the following Iitaka dimension
	\[
		\overline{\kappa}(X)=\kappa(V,K_V+D),
	\]
	which takes value in $\{-\infty,0,1,\dots,\dim X\}$.
	The characteristic property of $\kappa=\okappa(X)$ is:
	\[
		\alpha s^{\kappa}\leq\dim H^0(V,s(K_V+D))\leq\beta s^{\kappa}
	\]
	for some constants $0<\alpha<\beta$ and sufficiently large and divisible $s$.
	The definition of $\overline{\kappa}(X)$ is independent
	of the choice of the log smooth compactification.
	For details, see \cite[\S 11]{Iit82}.

	We use $\G_m$ to denote the $1$-dimensional \emph{algebraic torus},
	$\G_m^n$ the $n$-dimensional algebraic torus, and $\G_a$ the $1$-dimensional additive algebraic group.
	By a \emph{semiabelian variety} $X$, we mean the extension
	\[
		1\longrightarrow T\longrightarrow X\longrightarrow A\longrightarrow 1
	\]
	of an abelian variety $A$ by an algebraic torus $T$.
	A variety $X$ is called a \emph{\textit{Q}-algebraic torus}
	if there is a finite {\'e}tale cover $T\to X$ from an algebraic torus $T$.

	Let $X$ be a variety over $\kk$ and $f \colon X \dashrightarrow X$ a rational map.
	For any field extension $\kk \subseteq L$, denote by $(X_L, f_L)$ the base change of $(X,f)$ by $L$.
\end{setup}

\begin{lemma}\label{l:D_amp}
	Let $X$ be a smooth affine variety of dimension $d\geq 2$.
	Then we have:
	\begin{enumerate}[leftmargin=2em]
		\item Let $(V,D)$ be any compactification of $X$.
		      Then $D$ is the support of a connected big effective divisor.
		      Moreover, when $d=2$, $D$ is the support of a connected ample effective divisor.
		\item There is a log smooth compactification $(V,D)$ of $X$
		      such that $D$ is the support of a connected nef and big effective divisor.
	\end{enumerate}
\end{lemma}

\begin{proof}
	Let $(V,D)$ be any compactification of $X$.
	By~\cite[Ch.~\RomanNumeralCaps{2}, \S 3]{Har70}, $D$ has pure codimension $1$ and is connected.
	Embedding $X$ to an affine space $\A^n$ and taking its (normalised) closure in $\PP^n$,
	we can take a compactification $(V_0,D_0)$ of $X$
	such that $D_0$ is the support of an ample effective divisor $D_0'$
	of a normal projective variety $V_0$.
	Taking a log resolution of $(V_0,D_0)$,
	we obtain a log smooth compactification $(V',D')$ of $X$ such that
	$D'$ is the support of a connected effective nef and big divisor (the pullback of $D_0'$),
	which proves (2).
	Let $(V'',D'')$ be another log smooth compactification of $X$ dominating both $(V',D')$ and $(V,D)$.
	Then $D''$ is the inverse image of $D_0$ and hence the support of a connected effective big divisor,
	and $D$ is the image of $D''$ and hence also the support of a connected effective big divisor.
	For the case $\dim X=2$, see~\cite[Ch.~\RomanNumeralCaps{2}, \S 4]{Har70}.
\end{proof}

\begin{lemma}\label{l:exc}
	Let $V$ be a smooth projective surface with a $\PP^1$-fibration $\pi\colon V\to B$
	to a smooth projective curve $B$.
	Let $F$ be any fibre of $\pi$.
	If $F$ contains just one $(-1)$-curve,
	then its coefficient in $F$ is at least $2$.
\end{lemma}

\begin{proof}
	We remark that every fibre of the \(\PP^1\)-fibration $\pi$ consists of \(\PP^1\)'s
	and has a tree as its dual graph.
	Now the lemma follows by induction on the number of irreducible components of \(F\).
\end{proof}

\begin{definition}
	Let $X$ be a smooth quasi-projective surface
	and $\pi\colon X\to B$ a $\G_m$-bundle over a smooth curve $B$.
	We say that $\pi$ is \emph{untwisted} if some (and hence every) compactification
	$\overline{\pi}\colon\overline{X}\to\overline{B}$ of $\pi\colon X\to B$
	has exactly two cross-sections as the horizontal part of the boundary $\overline{X}\setminus X$.
\end{definition}

\begin{example}
	Take any Hirzebruch surface, with a ruling:
	\(\overline{\pi}\colon \overline{X}\to \overline{B}\)
	and let \(D \subseteq \overline{X}\) be an irreducible (resp.~reducible) curve
	so that there is a finite morphism \(\overline{\pi}|_D \colon D\to \overline{B}\) of degree \(2\).
	Let \(\{b_j\} \subseteq \overline{B}\) be a finite set such that
	\(\overline{\pi}|_D\) restricts to an \'etale cover
	\(H = D \setminus (\overline{\pi}|_D)^{-1}(\{b_j\}) \to B = \overline{B} \setminus \{b_j\}\).
	Let \(X = \overline{\pi}^{-1}(B) \setminus H\).
	Then \(\pi := \overline{\pi}|_X\colon X \to B\) is a twisted (resp.~untwisted) $\G_m$-bundle.
	Indeed, every other compactification of \(X \to B\) is
	obtained from this \(\overline{\pi}\) by blowing up or down fibre components
	(but not the horizontal boundary components),
	so the reducibility or irreducibility of the horizontal part
	will not change even in different compactification of \(X \to B\).
\end{example}

\begin{lemma}\label{l:cont}
	Let $X$ be a smooth quasi-projective surface
	and $\pi\colon X\to B$ a $\G_m$-bundle over a smooth curve $B$.
	Then we can take log smooth compactifications $X\subseteq\overline{X}$, $B\subseteq\overline{B}$
	such that $\pi$ extends to a $\PP^1$-bundle $\overline{\pi}\colon\overline{X}\to\overline{B}$.
	Moreover, if $\pi$ is untwisted,
	then we can take them such that $\overline{X}\setminus X$ consists of
	exactly two disjoint cross-sections and some fibres.
\end{lemma}

\begin{proof}
	Take log smooth compactifications $X\subseteq\overline{X}$,
	$B\subseteq\overline{B}$ such that
	$\pi$ extends to $\overline{\pi}\colon\overline{X}\to\overline{B}$.
	Take any point $b\in\overline{B}$ and set $F=\overline{\pi}^*b$.
	Suppose $F$ is reducible.
	If $b\in\overline{B}\setminus B$,
	then $F\subseteq\overline{X}\setminus X$,
	and we can contract $(-1)$-curves so that $F$ is irreducible.
	Assume $b\in B$.
	Let $F_0\cong\PP^1$ be the closure of $\pi^{-1}(b)$ in $\overline{X}$.
	Now $F$ contains at least one $(-1)$-curve $F_1$ since $\overline{\pi}$ is a $\PP^1$-fibration.
	If $F_1$ is the only $(-1)$-curve in $F$,
	then its multiplicity in $F$ is at least $2$ by \cref{l:exc}.
	So $F_1\neq F_0$ in this case.
	As a consequence, we have a $(-1)$-curve $F_1\neq F_0$ in $F$.
	We contract $F_1$ (without touching $X$) and continue this process;
	eventually we can make $F$ irreducible (and smooth).
	Thus we can assume $\overline{\pi}$ is a $\PP^1$-bundle.

	Assume that $\pi$ is untwisted.
	Let $H_1,H_2\subseteq\overline{X}\setminus X$ be two cross-sections.
	If the intersection $H_1\cap H_2\neq\varnothing$,
	it lies in $\overline{X}\setminus X$ since $\pi^{-1}(b)=\G_m$ for every $b\in B$.
	So we can make $H_1$ and $H_2$ disjoint
	by blowing up and down repeatedly in $\overline{X}\setminus X$.
\end{proof}

\begin{lemma}\label{l:smooth_K*}
	Let $X$ be a smooth quasi-projective surface
	and $\pi\colon X\to B$ a $\G_m$-bundle over a smooth curve $B$.
	Then we have:
	\begin{enumerate}[leftmargin=2em]
		\item There is a finite {\'e}tale cover $B'\to B$ of degree $\leq 2$
		      inducing a finite {\'e}tale cover $X'\coloneqq X\times_B B'\to X$
		      and a $\G_m$-bundle $X'\to B'$ which is untwisted.
		\item The log Kodaira dimensions satisfy
		      $\overline{\kappa}(X')=\overline{\kappa}(X)=\overline{\kappa}(B)=\overline{\kappa}(B')$.
		\item Assume that $\pi$ is untwisted and $B$ is a smooth rational affine curve.
		      Then $\pi$ is a trivial $\G_m$-bundle.
	\end{enumerate}
\end{lemma}

\begin{proof}
	(1) Take log smooth compactifications $X\subseteq\overline{X}$,
	$B\subseteq\overline{B}$ such that $\pi$
	extends to $\overline{\pi}\colon\overline{X}\to\overline{B}$.
	Let $D=\overline{X}\setminus X=H+E$
	where $H$ is the sum of horizontal components and $E$ is that of vertical components.
	If $H$ is irreducible,
	then $H\to\overline{B}$ restricts to a finite {\'e}tale cover $B'\to B$ of degree \(2\),
	since $X\to B$ is a $\G_m$-bundle.
	Let $\overline{B'}\to H$ be the normalisation,
	$X'=X\times_B B'$ and $\overline{X'}$ the normalisation of
	$\overline{X}\times_{\overline{B}}\overline{B'}$.
	Then the projection $X'\to X$ is a finite {\'e}tale cover,
	and the projection $X'\to B'$ which is still a $\G_m$-bundle will fit the next case.
	Indeed, the inverse of $H$ ($\subseteq\overline{X}$) in $\overline{X'}$
	is a double section containing a cross-section,
	the compactification of $\{(b',b')\in X'\mid b'\in B'\}$,
	so it is the sum of two cross-sections.

	(2) Since $X'\to X$ and $B'\to B$ are {\'e}tale,
	we have $\overline{\kappa}(X')=\overline{\kappa}(X)$
	and $\overline{\kappa}(B')=\overline{\kappa}(B)$ (cf.~\cite[Theorem 11.10]{Iit82}).
	From now on, we assume that $\pi\colon X\to B$ is untwisted.
	Let $H_1, H_2$ be cross-sections of $\overline{\pi}$ contained in $D$.
	By \cref{l:cont}, we may assume that $\overline{\pi}$ is a $\PP^1$-bundle and $H_1, H_2$ are disjoint.
	Then we have $K_{\overline{X}}+D\sim\overline{\pi}^*(K_{\overline B}+L)$
	where $L=\overline{B} \setminus B$
	(cf.~\cite[Ch.~\RomanNumeralCaps{5}, Proposition 2.9, Lemma 2.10]{Har77}).
	Thus $\overline{\kappa}(X)=\overline{\kappa}(B)$.

	(3) Assume $B=\PP^1\setminus(r\text{ points})$ with $r\geq 1$.
	Take log smooth compactifications $X\subseteq\overline{X}$, $B\subseteq\overline{B}$,
	as in \cref{l:cont},
	such that $\pi$ extends to a $\PP^1$-bundle $\overline{\pi}\colon\overline{X}\to\overline{B}$,
	and the boundary $D=\overline{X}\setminus X$ consists of
	exactly two disjoint cross-sections $H_1,H_2$ and $r$ of fibres.
	By blowing up a point in a fibre over $\overline{B}\setminus B$
	and blowing down proper transform of the fibre
	and repeating the same process if necessary,
	we can make $\overline{X}\cong\PP^1\times\PP^1$ without touching $X$,
	and yet keep $\overline{X}\setminus X$ being a union of two disjoint cross-sections and $r$ of fibres.
	Thus $\pi$ is a trivial $\G_m$-bundle.
\end{proof}

\begin{corollary}\label{c:smooth_K*}
	Let $X$ be a smooth affine surface and $\pi\colon X\to B$ a $\G_m$-bundle to a smooth curve $B$.
	Then $B$ is affine.
\end{corollary}

\begin{proof}
	By \cref{l:smooth_K*}, taking a finite {\'e}tale base change, we may assume that $\pi$ is untwisted.
	By \cref{l:cont}, we can embed $\pi$ to a $\PP^1$-bundle $\overline{X}\to\overline{B}$,
	such that $D\coloneqq\overline{X}\setminus X$ is the sum of two disjoint cross-sections
	and the fibres over $\overline{B}\setminus B$.
	If $\overline{B}=B$, then $D$ is not connected, a contradiction to \cref{l:D_amp}.
	So $B$ is not projective.
\end{proof}

\begin{definition}\label{d:mult_fib}
	Let $\pi\colon X \to B$ be a surjective morphism from a quasi-projective surface to a smooth curve.
	The \emph{multiplicity} $m=m(F)$ of a fibre $F$ of $\pi$
	is the greatest common divisor of coefficients in $F$ of all irreducible components of $F$.
	We call $F$ a {\it multiple fibre} if $m(F)\geq 2$.
\end{definition}

\begin{proposition}[cf.~{\cite[Lemma 1.1.9]{GMM21}}]\label{p:fat_fib}
	Let $X$ be a smooth quasi-projective surface
	and $\pi\colon X\to B$ a surjective morphism to a smooth curve $B$.
	Let $F_j$ $(1\leq j\leq r)$ be all multiple fibres of $\pi$ over $b_j$, with multiplicity $m_j\geq 2$.
	Assume the hypothesis $(\star)$:
	either $B$ is affine, or $B$ is irrational, or $r\geq 3$, or $r=2$ and $m_1=m_2$.
	Then we have:
	\begin{enumerate}[leftmargin=2em]
		\item There is a finite morphism $B'\to B$ from a smooth curve,
		      {\'e}tale over $B_0\coloneqq B\setminus\{b_1,\ldots,b_r\}$,
		      with ramification index $m_j$ over $b_j$,
		      and with $X'$ the normalisation of $X\times_B B'$,
		      such that the projection $X'\to X$ is finite {\'e}tale
		      and the induced fibration $X'\to B'$ has no multiple fibre.
		\item Assume further that every fibre of $\pi$ has support isomorphic to $\G_m$.
		      Then the fibration $X'\to B'$ is a $\G_m$-bundle over $B'$.
	\end{enumerate}
\end{proposition}

\begin{proof}
	By the solution to Fenchel's conjecture due to Fox,
	Bundgaard--Nielsen (cf.~\cite{BN51}, \cite{Cha83}),
	the hypothesis $(\star)$ implies the existence of a Galois covering $B'\to B$,
	ramified precisely over $b_j$ with index $m_j$ for $j=1,\dots,r$.
	Then the natural map $X'\to X$ from the normalisation $X'$ of $X\times_B B'$
	is {\'e}tale by the smoothness of $X$ and the purity of branch loci.
	This proves (1), while (2) follows from (1) and the assumption that $F_j=m_j[F_j]_{\red}$.
\end{proof}

It is now classical that on a $\Q$-factorial normal projective surface $V$,
every pseudo-effective $\Q$-divisor $L$ has the \emph{Zariski decomposition} $L=P+N$
to the sum of two $\Q$-Cartier divisors such that:
\begin{enumerate}
	\item $P$ is nef,
	\item either $N=0$ or $N$ is effective with $\Supp N=\bigcup N_i$ the irreducible decomposition
	      and with negative definite intersection matrix $(N_i\cdot N_j)$, and
	\item $(P\cdot N_i)=0$ for every irreducible component $N_i$ of $\Supp N$.
\end{enumerate}
We call $P$ (resp.~$N$) the \emph{nef part} (resp.~\emph{negative part}) of the decomposition.

The following result is well known.
We give a proof for the convenience of the reader.

\begin{proposition}\label{p:nef_free}
	Let $V$ be a $\Q$-factorial normal projective surface
	and $D$ an effective $\Q$-divisor such that $(V,D)$ has only log canonical singularities
	and $L\coloneqq K_V+D$ is pseudo-effective with $L=P+N$ its Zariski decomposition.
	Then $P$ is semi-ample and the inequality $P\leq L$ induces isomorphisms
	$H^0(V,\OO(sP))\cong H^0(V,\OO(sL))$ for all $s$ sufficiently large and divisible.
\end{proposition}

\begin{proof}
	By the known Minimal Model Program for log canonical surfaces (cf.~\cite{Fuj12}),
	there is a composition
	\[
		V=V_0\xrightarrow{\pi_0}V_1\xrightarrow{\pi_1}\cdots\xrightarrow{\pi_{m-1}}V_m\eqqcolon W
	\]
	of birational contractions of $(K_{V_i}+D_i)$-negative extremal curves $E_i\subseteq V_i$
	such that $K_W+D_W$ is nef and hence semi-ample by the abundance theorem (cf.~\cite[(3.13)]{KM98});
	here $D_i\subseteq V_i$ is the direct image of $D$, and $D_W\coloneqq D_m$;
	further, $K_{V_i}+D_i=\pi_i^*(K_{V_{i+1}}+D_{i+1})+a_iE_i$ for some $a_i>0$.
	Set $\pi\coloneqq\pi_{m-1}\circ\cdots\circ\pi_0\colon V\to W$.
	Then
	\[
		L=K_V+D=\pi^*(K_W+D_W)+E
	\]
	where $E$ is effective and $\pi$-exceptional.
	Thus $P=\pi^*(K_W+D)$ and $N=E$.
	The second isomorphism below follows from the projection formula:
	\begin{align*}
		H^0(V,\OO(sL)) & =H^0(W,\pi_*\OO(sL)) \cong
		H^0(W,\OO(s(K_W+D_W))\otimes\pi_*\OO(sE))   \\           & =
		            H^0(W,\OO(s(K_W+D_W))=H^0(V,\OO(sP)))
	\end{align*}
	since $E\geq 0$ is $\pi$-exceptional.
\end{proof}

\begin{lemma}[cf.~{\cite[Theorem 11.6 and Theorem 11.12]{Iit82}}]\label{l:finord}
	Let $X$ be a variety and $f\colon X\to X$ a dominant morphism.
	Assume $\overline{\kappa}(X)=\dim X$.
	Then $f$ is an automorphism of finite order.
\end{lemma}

\begin{lemma}[cf.~{\cite[Theorem 2]{Iit77}}]\label{lem:k>0_et}
	Let $X$ be a smooth variety with $\overline{\kappa}(X)\geq 0$.
	Then any dominant morphism from $X$ into itself is an {\'e}tale surjective morphism.
\end{lemma}

\begin{lemma}\label{l:euler}
	Let $X$ be a smooth variety and $f\colon X\to X$ a dominant morphism.
	Suppose that $f$ is {\'e}tale and surjective
	(this is the case when $\overline{\kappa}(X)\geq 0$; cf.~\cref{lem:k>0_et}).
	Suppose further that $\deg(f)\geq 2$.
	Then the topological Euler number $e(X)$ of $X$ is $0$.
\end{lemma}

\begin{proof}
	We have $e(X)=\deg(f)\cdot e(X)$ and $\deg(f)\geq 2$, so $e(X)=0$.
\end{proof}

\begin{lemma}\label{lem:embd_in_P1xP1}
	Any $\A^1$-bundle over a smooth rational affine curve $B$ is trivial.
\end{lemma}

\begin{proof}
	The $\A^1$-bundles over $B$ are classified by $H^1_{\acute{e}t}(B,G)$
	where $G=\Aut(\A^1)\cong\G_a\rtimes\G_m$.
	Since $B\subseteq\A^1$ and $\Pic(\A^1)=0$, we have $H^1_{\acute{e}t}(B,\G_m)=\Pic(B)=0$.
	Consider the short exact sequence $0\to\G_a\to G\to\G_m\to 1$
	and the fact that $H^1_{\acute{e}t}(B,\G_a)=H^1_{\acute{e}t}(B,\G_m)=0$.
	We have $H^1_{\acute{e}t}(B,G)=0$, which concludes the proof.
\end{proof}

\begin{lemma}\label{l_inv_fun_field}
	Let $K$ be an algebraically closed subfield of $\kk$.
	Let $X$ be a variety over $K$ and $f\colon X\dashrightarrow X$ a dominant rational map.
	Then the following statements are equivalent:
	\begin{enumerate}
		\item \(K(X)^{f}\neq K\);
		\item \(K(X)^{f^{\ell}}\neq K\) for some (and hence for all) \(\ell\geq 1\); and
		\item \(\kk(X_{\kk})^{f_{\kk}}\neq \kk\).
	\end{enumerate}
\end{lemma}

\begin{proof}
	By \cite[Lemma 2.1]{Xie19}, (1) and (2) are equivalent.
	The direction $(1)\Rightarrow (3)$ is easy.
	We only need to prove $(3) \Rightarrow (1)$.
	Pick $\phi\in\kk(X_{\kk})^{f_{\kk}}\setminus\kk$.
	View $\kk$ as a $K$-vector space with basis $\{b_i\}_{i\in I}$.
	Note that $\kk(X_{\kk})\cong\kk\otimes_{K}K(X)$.
	So we may write $\phi=\sum_{i\in I}\phi_ib_i$ where $\phi_i\in K(X)$.
	Then $\phi=(f_{\kk})^*\phi=\sum_{i\in I}(f^*\phi_i)b_i$,
	which implies $f^*\phi_i=\phi_i$ for all $i\in I$.
	Since $\phi$ is not a constant, some $\phi_i$ is nonconstant in $K(X)$ and hence $K(X)^f\neq K$.
\end{proof}

We need the following result from \cite[Expos{\'e} \RomanNumeralCaps{13}, Proposition 4.6]{GR71}.

\begin{lemma}\label{lem:pi_1-inv}
	Let $\kk$ be an algebraically closed field of characteristic zero,
	and $X$ a normal quasi-projective variety over $\kk$.
	Let $L$ be any algebraically closed field with an injective homomorphism $\kk\to L$.
	Then, the natural map $\pi_1^{\et}(X_L)\to\pi_1^{\et}(X)$ is an isomorphism.
\end{lemma}

\begin{lemma}\label{lem:extsemiabel}
	Let $K$ be an algebraically closed subfield of $\kk$.
	Let $X$ be a variety over $K$ such that $X_{\kk}$ is
	isomorphic to a semiabelian variety (resp.~an algebraic torus).
	Then $X$ is isomorphic to a semiabelian variety (resp.~an algebraic torus).
\end{lemma}

\begin{proof}
	We only prove the semiabelian variety case.
	The algebraic torus case is similar.

	By assumption, there is an isomorphism $\phi_{\kk}\colon X\times_{\Spec K}\Spec\kk\to G$,
	where $G$ is a semiabelian variety over $\kk$ (cf.~\ref{NT}).
	Let $L$ be a subfield of $\kk$ such that $G$ (and its group structure)
	and $\phi_{\kk}$ are defined over $L$ and $L$ is finitely generated over $K$,
	i.e., there is a semiabelian variety $G_L$ over $L$
	and an isomorphism $\phi_L\colon X\times_{\Spec K}\Spec L\to G_L$.
	Then there exists an affine variety $B$ over $K$ such that $K(B)=L$,
	a semiabelian scheme $G_B\to B$ over $B$ whose generic fibre is $G_L$.
	After shrinking $B$,
	there is an isomorphism of $B$-schemes $\phi_B\colon X\times_{\Spec K} B\to G_B$
	whose restriction to the generic fibre is $\phi_L$.
	Picking $b\in B(K)$, the restriction of $\phi_B$ to the fibre over $b$
	gives an isomorphism between $X$ and $G_b$.
	Here $G_b$ is the fibre of $G_B\to B$ over $b$, which is a semiabelian variety over $K$.
\end{proof}

\begin{remark}\label{r:red_2C}
	Let $P(L)$ be a property of algebraic varieties
	and morphisms over an algebraically closed field $L$ of characteristic zero.
	Assume that:
	\begin{align}\label{eq:property}
		\tag{$\diamondsuit$}
		       & \text{for algebraically closed fields $\kk\subseteq\kk'$,}   \\
		\notag & \text{$P(\kk)$ holds true if and only if so does $P(\kk')$.}
	\end{align}
	Then for a fixed algebraically closed field $\kk$, $P(\kk)$ holds true if and only if so does $P(\C)$.

	Indeed, we may assume that the algebraic varieties
	and morphisms are defined over a subfield $K\subseteq \kk$ which is finitely generated over $\Q$.
	We then fix an embedding $K\subseteq\C$.
	Let $\kk_1\subseteq\C$ (resp.~$\kk_2\subseteq\kk$)
	be the algebraic closure of $K$ in $\C$ (resp.~$\kk$).
	Identifying the algebraic varieties and morphisms over $\kk_1$ and $\kk_2$
	with the isomorphism induced from an isomorphism $\kk_1\to \kk_2$.
	Then the claim follows from \eqref{eq:property}.
\end{remark}

\begin{lemma}\label{l:etl_pre_torus}
	Let $T$ be a semiabelian variety (resp.~an algebraic torus)
	and $\tau\colon\widetilde{T}\to T$ a finite {\'e}tale cover.
	Then $\widetilde{T}$ is also a semiabelian variety (resp.~an algebraic torus).
\end{lemma}

\begin{proof}
	We may assume $\kk=\C$ (cf.~\cref{lem:pi_1-inv,lem:extsemiabel} and \cref{r:red_2C}).
	Consider first the semiabelian variety case.
	Note that $\overline{\kappa}(\widetilde{T})=\overline{\kappa}(T)=0$.
	By the universal property of the quasi-Albanese map $a\colon\widetilde{T}\to A$
	(cf.~\cite{Iit76}), $\tau$ factors through $a$.
	In particular, $a$ is a finite surjective morphism.
	Moreover, $a$ has irreducible general fibres (cf.~\cite[Theorem~28]{Kaw81}).
	Thus $a$ is an isomorphism and hence $X$ is a semiabelian variety.

	When $T$ is further an algebraic torus, it is affine, so is $\widetilde{T}$.
	Then $\widetilde{T}$ being semiabelian and affine implies that it is an algebraic torus.
\end{proof}

\begin{lemma}[cf.~{\cite[Lemma~2.12]{NZ10}}]\label{lem:torus-cls}
	Let $X$ be a \textit{Q}-algebraic torus.
	Then there is a finite {\'e}tale cover $\pi_T\colon T\to X$ such that the following hold.
	\begin{enumerate}[leftmargin=2em]
		\item $T$ is an algebraic torus, and $\pi_T$ is Galois.
		\item If there is another finite {\'e}tale cover $\pi_{T'}\colon T'\to X$
		      from an algebraic torus $T'$,
		      then there is an {\'e}tale morphism $\tau\colon T'\to T$
		      such that $\pi_{T'}=\pi_T\circ\tau$.
	\end{enumerate}
	We call $\pi_T$ the algebraic torus closure of $X$.
\end{lemma}

\begin{proof}
	Since $X$ is a \textit{Q}-algebraic torus,
	there is a finite {\'e}tale morphism $\pi_T\colon T\to X$ where $T$ is an algebraic torus.
	After taking its Galois closure, we may assume that $\pi_T$ is Galois (cf.~\cref{l:etl_pre_torus}).
	Then $X=T/G_T$ where $G_T$ is a finite subgroup of $\Aut_{var}(T)$
	(the automorphism group of the variety $T$).
	Let $G_0=G_T\cap\{\text{translations on }T\}$.
	Then $T/G_0\to X$ is {\'e}tale and Galois, and $T/G_0$ is an algebraic torus.
	So we may assume $G_T$ is translation-free.
	We next show that $\pi_T$ satisfies the universal property (2).

	Suppose that there is another finite {\'e}tale cover $\pi_{T'}\colon T'\to X$
	from an algebraic torus $T'$.
	By taking the base change and the Galois closure,
	there exist {\'e}tale morphisms $\widetilde{T}\to T$ and $\widetilde{T}\to T'$ over $X$
	such that the composition $\widetilde{T}\to X$ is Galois.
	Clearly, $\widetilde{T}$ is an algebraic torus (cf.~\cref{l:etl_pre_torus}).
	Then $X\cong\widetilde{T}/G_{\widetilde{T}}$ and $T\cong\widetilde{T}/H_{T}$
	where $G_{\widetilde{T}}$ is a finite subgroup of $\Aut_{var}(\widetilde{T})$
	and $H_T=\Gal(\widetilde{T}/T)$ is a subgroup of $G_{\widetilde{T}}$.
	Similarly, $T'\cong\widetilde{T}/H_{T'}$
	where $H_{T'}=\Gal(\widetilde{T}/T')$ is a subgroup of $G_{\widetilde{T}}$.
	Since $T$ and $T'$ are both algebraic tori,
	$H_T$ and $H_{T'}$ are translation subgroups of $\widetilde{T}$.
	By our construction of $T$,
	the group $H_T=G_{\widetilde{T}}\cap\{\text{translations on }\widetilde{T}\}$.
	So $H_{T'}$ is a subgroup of $H_T$.
	Hence there is a natural {\'e}tale morphism $\tau\colon T'\to T=T'/(H_{T}/H_{T'})$
	such that $\pi_{T'}=\pi_T\circ\tau$.
\end{proof}

\section{Dynamical degrees and arithmetic degrees}\label{dyn-arith}

In this section, we will upper-bound arithmetic degree by dynamical degree,
and show that the arithmetic degree (like dynamical degree)
is preserved by generically finite morphisms;
see \cref{proupboundarth,p:kscequ}.

\begin{setup}{\textbf{The dynamical degrees.}}\label{dyn_deg}
	In this part, we work over an algebraically closed field of arbitrary characteristic.
	Let $X$ be a variety and $f\colon X\dashrightarrow X$ a dominant rational self-map.
	Let $X'$ be a normal projective variety which is birational to $X$.
	Let $L$ be an ample (or just nef and big) divisor on $X'$.
	Denote by $f'$ the rational self-map of $X'$ induced by $f$.
	For $i=0,1,\dots,\dim X$, and $n\geq 0$,
	define $(f'^n)^*(L^i)$ to be the $(\dim X-i)$-cycle on $X'$ as follows:
	let $\Gamma$ be a normal variety with a birational morphism $\pi_1\colon\Gamma\to X'$
	and a morphism $\pi_2\colon\Gamma\to X'$
	such that $f'^n=\pi_2\circ\pi_1^{-1}$.
	Then $(f'^n)^*(L^i)\coloneqq (\pi_1)_*\pi_2^*(L^i)$.
	The definition of $(f'^n)^*(L^i)$ does not depend on the choice of $\Gamma$, $\pi_1$ and $\pi_2$.
	Then
	\[
		d_i(f)\coloneqq\lim_{n\to\infty}((f'^n)^*(L^i)\cdot L^{\dim X-i})^{1/n}
	\]
	is called the $i$-th \emph{dynamical degree} of $X$.
	The limit converges and does not depend on the choice of $X'$ and $L$;
	moreover, if $\pi\colon X\dashrightarrow Y$ is a generically finite and dominant rational map
	between varieties and $g\colon Y\dashrightarrow Y$ is a rational self-map
	such that $g\circ\pi=\pi\circ f$, then $d_i(f)=d_i(g)$ for all $i$;
	for details, we refer to \cite[Theorem~1]{Dan20} (and the projection formula),
	or Theorem~4 in its arXiv version.
\end{setup}

\Cref{p:dyn_sub_var} below is easy when $\kk$ is of characteristic $0$ and $Z\not\subseteq\Sing X$.

\begin{proposition}\label{p:dyn_sub_var}
	Let $X$ be a variety over an algebraically closed field $\kk$ of arbitrary characteristic,
	and $f\colon X\dashrightarrow X$ a dominant rational self-map.
	Denote by $I(f)$ the indeterminacy locus of $f$.
	Let $Z$ be an irreducible subvariety in $X$ which is not contained in $I(f)$
	such that $f|_Z$ induces a dominant rational self-map of $Z$.
	Then $d_i(f|_Z)\leq d_i(f)$ for $i=0,1,\dots,\dim Z$.
\end{proposition}

\begin{proof}
	Set $d_X\coloneqq\dim X$ and $d_Z\coloneqq\dim Z$.
	Denote by $\eta_Z$ the generic point of~$Z$.

	\medskip

	We first reduce to the projective case.
	Let $U$ be an affine open subset of $X$ containing $\eta_Z$
	and $X'$ a projective compactification of $U$.
	Denote by $f'$ the rational self-map of $X'$ induced by $f$.
	Note that $f$ is well-defined at $\eta_Z$ and $f(\eta_Z)=\eta_Z$.
	Similarly $f'$ is well-defined at $\eta_Z$ and $f'(\eta_Z)=\eta_Z$.
	Let $Z'$ be the Zariski closure of $\eta_Z$ in $X'$.
	Then $Z'\not\subseteq I(f')$ and $f'|_{Z'}$ induces a dominant rational self-map of $Z'$.
	We have $d_i(f)=d_i(f')$ and $d_j(f|_Z)=d_j(f'|_{Z'})$ for all $i, j$ (cf.~e.g.,~\cite{Dan20}).
	After replacing $X,f,Z$ by $X',f',Z'$, we may assume that $X$ is projective.

	\medskip

	Next we reduce to the normal case.
	Let $\pi\colon X'\to X$ be the normalisation of $X$.
	Set $Z'\coloneqq\pi^{-1}(Z)$, the set-theoretic preimage of $Z$.
	It is reduced, but may not be irreducible.
	Let $f'$ be the rational self-map of $X'$ induced by $f$.
	Since $\pi$ is finite and $I(f')\subseteq\pi^{-1}(I(f))$,
	$f'$ is well-defined at every generic point of $Z'$.
	Since $f(\eta_Z)=\eta_Z$ and $\pi^{-1}(\eta_Z)$ is finite,
	there is some $m\geq 1$ and $\eta_{Z_1}\in \pi^{-1}(\eta_Z)$
	such that $(f')^m(\eta_{Z_1})=\eta_{Z_1}$.
	Let $Z_1$ be the Zariski closure of $\eta_{Z_1}$.
	Then $\pi(Z_1)=Z$ and $\pi|_{Z_1}\colon Z_1\to Z$ is finite.
	Observe that $(f')^m|_{Z_1}$ induces a dominant rational self-map of $Z_1$.
	We have $d_i((f')^m|_{Z_1})=d_i(f|_Z)^m$ for all $i$.
	After replacing $X,f,Z$ by $X',f'^m,Z_1$, we may assume that $X$ is normal.

	For $n\geq 0$, consider the following commutative diagram.
	\[
		\xymatrixcolsep{3.5pc}
		\xymatrix{
		\widetilde{Z_n}\ar[d]_{\widetilde{\pi_1^n|_{Z_n}}}\ar[r]^{p_n}
		& Z_n\ar[d]_{\pi_1^n|_{Z_n}}\ar@{^{(}->}[r]	& \Gamma_n\ar[d]_{\pi_1^n}\ar[dr]^{\pi_2^n}\\
		\widetilde{Z}\ar[r]_p
		& Z \ar@{^{(}->}[r]& X \ar@{-->}[r]_{f^n}	& X
		}
	\]
	Here $\Gamma_n$ is a normal projective variety;
	the map $\pi_1^n$ is a birational morphism satisfying $I((\pi_1^n)^{-1})\subseteq I(f^n)$;
	the map $\pi_2^n$ is a morphism satisfying $f^n=\pi_2^n\circ(\pi_1^n)^{-1}$;
	the variety $Z_n$ is the strict transform of $Z$ under $(\pi_1^{n})^{-1}$, i.e.,
	\[
		Z_n=\overline{(\pi_1^n)^{-1}(Z\setminus I((\pi_1^n)^{-1}))}
		=\overline{(\pi_1^n)^{-1}(Z\setminus I(f^n))};
	\]
	the maps $p\colon\widetilde{Z}\to Z$ and $p_n\colon\widetilde{Z_n}\to Z_n$ are the normalisations.
	Since $Z$ is $f^n$-invariant, we get the following commutative diagram:
	\[
		\xymatrixcolsep{4pc}
		\xymatrix{
		\widetilde{Z_n}\ar[d]_{\widetilde{\pi_1^n|_{Z_n}}}\ar[dr]^{\widetilde{\pi_2^n|_Z}}\ar[rr]^{p_n}
		& {} & Z_n\ar[d]_{\pi_1^n|_{Z_n}}\ar[dr]^{\pi_2^n|_Z} \\
		\widetilde{Z}\ar@/_27pt/[rr]^p \ar@{-->}[r]_{(\widetilde{f|_Z})^n}
		& \widetilde{Z}\ar@/_27pt/[rr]^p & Z\ar@{-->}[r]_{(f|_Z)^n}& Z
		.}
	\]
	Here $\widetilde{\pi_1^n|_{Z_n}}$ (as in the above diagram too) is induced by $\pi_1^n|_{Z_n}$, the map $\widetilde{\pi_2^n|_{Z_n}}$ is induced by $\pi_2^n|_{Z_n}$ and $\widetilde{f|_Z}$ is induced by $f|_Z$.
	Let $L$ be an ample divisor on $X$.

	\medskip

	For $i=0,1,\dots,d_Z$, by the projective formula, we have:
	\begin{align*}
		d_i(f)    & =\lim_{n\to\infty}((f^n)^*(L^i)\cdot L^{d_X-i})^{1/n}=\lim_{n\to\infty}((\pi_2^n)^*(L^i)\cdot(\pi_1^n)^*(L^{d_X-i}))^{1/n}; \\
		d_i(f|_Z) & =d_i(\widetilde{f|_{Z}})=\lim_{n\to\infty}(((\widetilde{f|_Z})^n)^*(p^*(L^i))\cdot p^*(L^{d_Z-i}))^{1/n}                    \\
		          & =\lim_{n\to \infty}(\widetilde{\pi_2^n|_{Z_n}}^*(p^*(L^i))\cdot\widetilde{\pi_1^n|_{Z_n}}^*(p^*(L^{d_Z-i})))^{1/n}          \\
		          & =\lim_{n\to \infty}(p_n^*((\pi_2^n)^*(L^i))\cdot p_n^*((\pi_1^n)^*(L^{d_Z-i})))^{1/n}                                       \\
		          & =\lim_{n\to \infty}((\pi_2^n)^*(L^i)\cdot (\pi_1^n)^*(L^{d_Z-i})\cdot Z_n)^{1/n}.
	\end{align*}
	After replacing $L$ by its positive multiple,
	we may assume that $L$ is very ample and there are $H_1,\dots,H_{d_X-d_Z}\in |L|$
	such that $\bigcap_{j=1}^{d_X-d_Z}H_j$ is of pure dimension equal to $\dim Z$
	and it contains $Z$ as an irreducible component.
	Then we have:
	\begin{align}
		 & d_i(f)     =\lim_{n\to\infty}((\pi_2^n)^*(L^i)\cdot(\pi_1^n)^*(L^{d_Z-i})\cdot(\pi_1^n)^*H_1\cdot\dots\cdot(\pi_1^n)^*H_{d_X-d_Z})^{1/n},\hspace{-1ex}\label{equdif} \\
		 & d_i(f|_Z)  =\lim_{n\to\infty}((\pi_2^n)^*(L^i)\cdot(\pi_1^n)^*(L^{d_Z-i})\cdot Z_n)^{1/n}.\label{equdifZ}
	\end{align}

	We need the following:

	\begin{lemma}\label{lemsuppin}
		Let $Y$ be a normal projective variety,
		$L_1,\dots, L_r$ effective and nef Cartier divisors on $Y$.
		Let $W$ be an irreducible component of $\bigcap_{j=1}^rL_j$,
		which is of codimension $r$ in $Y$
		(but $\bigcap_{j=1}^rL_j$ is not assumed to be of pure dimension).
		Then we have $W\leq L_1 \cdot\cdots\cdot L_r\in \N^r(Y)$,
		i.e., $L_1\cdot\cdots\cdot L_r-W$ is pseudo-effective in $\N^r(Y)$
		(the real vector space of codimension-$r$ cycle classes modulo numerical equivalence).
	\end{lemma}

	We return back to the proof of \cref{p:dyn_sub_var}.
	Since $\bigcap_{j=1}^{d_X-d_Z}H_j$ is of pure codimension $d_X-d_Z$,
	it contains $Z$ as an irreducible component,
	$Z_n$ is the strict transform of $Z$ under $(\pi^n_1)^{-1}$,
	and $Z_n\subseteq\bigcap_{j=1}^{d_X-d_Z}(\pi^n_1)^*H_j$.
	By \cref{lemsuppin}, we get
	$Z_n\leq (\pi_1^n)^*H_1\cdot\cdots\cdot (\pi_1^n)^*H_{d_X-d_Z}$.
	Since both $(\pi_1^n)^*L$ and $(\pi_2^n)^*L$ are nef, we get
	\begin{multline*}
		(\pi_2^n)^*(L^i)\cdot (\pi_1^n)^*(L^{d_Z-i})\cdot Z_n\\
		\leq (\pi_2^n)^*(L^i)\cdot (\pi_1^n)^*(L^{d_Z-i})\cdot (\pi_1^n)^*H_1\cdot\cdots\cdot (\pi_1^n)^*H_{d_X-d_Z}.
	\end{multline*}
	Applying this to Equations \eqref{equdif} and \eqref{equdifZ},
	we get $d_i(f|_Z)\leq d_i(f)$.
	This proves \cref{p:dyn_sub_var} modulo \cref{lemsuppin}.

	\medskip

	We still have to prove \cref{lemsuppin}, by induction on $r$.
	The case $r=1$ is clear.

	Now we assume that $r\geq 2$ and \cref{lemsuppin} holds for $r-1$.
	After relabelling, there is an irreducible component $W_{r-1}$ of $\bigcap_{j=1}^{r-1}L_j$
	of codimension $r-1$ such that $W_{r-1}\not\subseteq L_r$ and $W\subseteq W_{r-1}\cap L_r$.
	Then $W_{r-1}\cap L_r$ is of pure codimension $r$ and $W$ is an irreducible component of it.
	Thus $W\leq W_{r-1}\cdot L_r$.
	By the induction hypothesis, we get $W_{r-1}\leq L_1\cdot\cdots\cdot L_{r-1}$.
	Since $L_{r}$ is nef, we have
	\[
		W\leq W_{r-1}\cdot L_r\leq L_1\cdot\cdots\cdot L_{r-1}\cdot L_r.
	\]
	This proves \cref{lemsuppin}, and also completes the proof of \cref{p:dyn_sub_var}.
\end{proof}

\begin{setup}\textbf{Admissible triples.}
	We define an \emph{admissible triple} to be $(X,f,x)$
	where $X$ is a quasi-projective variety over $\overline{\Q}$,
	$f\colon X\dashrightarrow X$ is a dominant rational self-map
	and $x\in X(\overline{\Q})$ such that $f$ is well-defined at $f^n(x)$, for any $n\geq 0$.

	We say that $(X,f,x)$ \emph{dominates} (resp.~\emph{generically finitely dominates}) $(Y,g,y)$
	if there is a dominant rational map (resp.~generically finite and dominant rational map)
	$\pi\colon X\dashrightarrow Y$ such that $\pi\circ f=g\circ\pi$,
	$\pi$ is well defined along $O_f(x)$ and $\pi(x)=y$.

	We say that $(X,f,x)$ is \emph{birational} to $(Y,g,y)$
	if there is a birational map $\pi\colon X\dashrightarrow Y$
	such that $\pi\circ f=g\circ\pi$
	and if there is a Zariski dense open subset $V$ of $Y$ containing $O_g(y)$
	such that $\pi|_U\colon U\coloneqq\pi^{-1}(V)\to V$ is a well-defined isomorphism and $\pi(x)=y$.
	In particular, if $(X,f,x)$ is birational to $(Y,g,y)$,
	then $(X,f,x)$ generically finitely dominates $(Y,g,y)$.
\end{setup}

\begin{remark}
	\leavevmode
	\begin{enumerate}
		\item If $(X,f,x)$ dominates $(Y,g,y)$ and if $O_f(x)$ is Zariski dense in $X$,
		      then $O_g(y)$ is Zariski dense in $Y$.
		      Moreover, if $(X,f,x)$ generically finitely dominates $(Y,g,y)$,
		      then $O_f(x)$ is Zariski dense in $X$ if and only if $O_g(y)$ is Zariski dense in $Y$.
		\item Every admissible triple $(X,f,x)$ is birational to
		      an admissible triple $(X',f',x')$ where $X'$ is projective.
		      Indeed, we may pick $X'$ to be any projective compactification of $X$,
		      $f'$ the self-map of $X'$ induced from $f$, and $x'=x$.
	\end{enumerate}
\end{remark}

\begin{lemma}\label{lembirinverse}
	Let $\pi\colon X\dashrightarrow Y$ be a birational map between projective varieties.
	Let $U$ be a open subset of $X$.
	If $\pi$ is well defined on $U$, $V\coloneqq \pi(U)$ is open in $Y$
	and $\pi|_U\colon U\to V$ is an isomorphism,
	then $\pi^{-1}$ is well defined on $V$ and $\pi^{-1}(V)=U$.
\end{lemma}

\begin{proof}
	There are birational morphisms $\pi_1\colon Z\to X$ and $\pi_2\colon Z\to Y$
	such that $\pi=\pi_2\circ\pi_1^{-1}$ and $\pi_1$ is an isomorphism on $\pi_1^{-1}(U)$.
	If \cref{lembirinverse} holds for $\pi_2$, then it holds for $\pi$.
	After replacing $X,\pi, U$ by $Z, \pi_2, \pi_1^{-1}(U)$, we may assume that $\pi$ is a morphism.

	Let $\widetilde{U}, \widetilde{V}, \widetilde{X},\widetilde{Y}$ be the normalisations of $U,V,X,Y$,
	and $\widetilde{\pi}\colon\widetilde{X}\to \widetilde{Y}$ the morphism induced by $\pi$.
	Then $\widetilde{U}$ is open in $\widetilde{X}$,
	$\widetilde{V}$ is open in $\widetilde{Y}$,
	$\widetilde{\pi}(\widetilde{U})=\widetilde{V}$
	and $\widetilde{\pi}|_{\widetilde{U}}\colon \widetilde{U}\to \widetilde{V}$ is an isomorphism.
	If \cref{lembirinverse} holds for $\widetilde{\pi}$, then it also holds for $\pi$.
	So we may assume that $X$ and $Y$ are normal.

	For every $y\in V$, pick $x\in U$, such that $\pi(x)=y$.
	Then $\{x\}=U\cap \pi^{-1}(y)$. So $x$ is an isolated point in $\pi^{-1}(y)$.
	By Zariski's main theorem $\pi^{-1}(y)$ is connected.
	So $\pi^{-1}(y)=\{x\}$.
	Then $\pi^{-1}(V)=U$.
\end{proof}

\begin{setup}\textbf{The set $A_f(x)$.}
	When $X$ is projective and $L$ is a Cartier divisor on $X$,
	denote by $h_L\colon X(\overline{\Q})\to\R$ a Weil height function on $X$ associated to $L$.
	It is unique up to adding a bounded function.
	When we have a morphism $\pi\colon X\to Y$,
	for a Cartier divisor $M$ on $Y$,
	we may choose $h_{\pi^*M}$ to be $h_M\circ\pi$.
	To simplify the notations, in this section, we always make this choice without saying it.

	For a projective admissible triple $(X,f,x)$,
	let $L$ be an ample divisor on $X$,
	we define
	\[
		A_f(x)\subseteq [0,\infty]
	\]
	to be the limit set of the sequence $(h_L^+(f^n(x)))^{1/n}, n\geq 0$
	where $h_L^+(\cdot)\coloneqq\max\{h_L(\cdot),1\}$.

	The following lemma shows that the set $A_f(x)$ does not depend on the choice of $L$
	and is invariant in the birational equivalence class of $(X,f,x)$.
\end{setup}

\begin{lemma}\label{lemsingwilldef}
	Let $\pi\colon X\dashrightarrow Y$ be a dominant rational map between projective varieties.
	Let $U$ be a Zariski dense open subset of $X$ such that $\pi|_U\colon U\to Y$ is well-defined.
	Let $L$ be an ample divisor on $X$ and $M$ an ample divisor on $Y$.
	Then there are constants $C\geq 1$ and $D>0$ such that for every $x\in U$, we have
	\begin{equation}\label{equationdomineq1}
		h_M(\pi(x))\leq Ch_L(x)+D.
	\end{equation}

	Moreover if $V\coloneqq\pi(U)$ is open in $Y$ and $\pi|_U\colon U\to V$ is an isomorphism,
	then there are constants $C\geq 1$ and $D>0$ such that for every $x\in U$, we have
	\begin{equation}\label{equationbirdomineq}
		C^{-1}h_L(x)-D\leq h_M(\pi(x))\leq Ch_L(x)+D.
	\end{equation}
\end{lemma}

\begin{proof}
	There is a birational morphism $\pi_1\colon Z\to X$ and a dominant morphism $\pi_2\colon Z\to Y$
	such that $\pi=\pi_2\circ\pi_1^{-1}$ and $\pi_1$ is an isomorphism on $\pi_1^{-1}(U)$.
	If \cref{lemsingwilldef} holds for $\pi_1$ and $\pi_2$, it holds for $\pi$.
	So we may assume that $\pi$ is a morphism.

	For Inequality~\eqref{equationdomineq1},
	we may assume that $CL-\pi^*M$ is ample on $X$ for some integer $C\geq 1$.
	Then there is a constant $D>0$ such that for every $x\in U$,
	\[
		h_M(\pi(x))=h_{\pi^*M}(x)\leq Ch_L(x)+D.
	\]

	Next we prove Inequality~\eqref{equationbirdomineq}
	and hence assume that $\pi$ is a birational morphism.
	By Inequality~\eqref{equationdomineq1},
	we only need to prove the first part of Inequality~\eqref{equationbirdomineq}.
	Assume that $V$ is open in $Y$ and $\pi|_U\colon U\to V$ is an isomorphism.
	By \cref{lembirinverse}, $U=\pi^{-1}(V)$.

	Set $\mathcal{F}(m)=\pi^*\mathcal{O}_Y(mM)\otimes\mathcal{O}_X(-L)$.
	Then $\pi_*\mathcal{F}(m)\cong\mathcal{O}_Y(mM)\otimes\pi_*\mathcal{O}_X(-L)$
	is globally generated for $m\gg 1$.
	We have a surjective morphism of sheaves $\phi\colon\mathcal{O}_Y^{\oplus r}\to\pi_*\mathcal{F}(m)$.
	Pulling back by $\pi$,
	we have $\psi\colon\mathcal{O}_X^{\oplus r}\to\pi^* \pi_*\mathcal{F}(m)\to\mathcal{F}(m)$.

	Now $(\pi^*\pi_*\mathcal{F}(m))|_U\cong\mathcal{F}(m)|_U$,
	so the restriction of $\psi$ to $U$ gives a surjective morphism
	$\psi|_U=\pi^*\phi|_U\colon\mathcal{O}_U^{\oplus r}\to\mathcal{F}(m)|_U$.
	This implies that $m\pi^*M-L$ has its base locus outside $U$.
	Hence there is a constant $D>0$ such that for every $x\in U$, $h_M(\pi(x))\geq m^{-1}h_L(x)-D$.
\end{proof}

\begin{setup}\textbf{The arithmetic degree.}\label{subsec_arithdeg}
	More generally, for every admissible triple $(X,f,x)$, we define $A_f(x)$ to be $A_{f'}(x')$
	where $(X',f',x')$ is an admissible triple
	which is birational to $(X,f,x)$ such that $X'$ is projective.
	By \cref{lemsingwilldef}, this definition does not depend on the choice of $(X',f',x')$.
	We define (see also \cite{KS16}):
	\[
		\overline{\alpha}_f(x)\coloneqq\sup A_f(x),\qquad\underline{\alpha}_f(x)\coloneqq\inf A_f(x).
	\]
	We say that $\alpha_f(x)$ is well-defined and call it the \emph{arithmetic degree} of $f$ at $x$,
	if $\overline{\alpha}_f(x)=\underline{\alpha}_f(x)$;
	and, in this case, we set
	\[
		\alpha_f(x)\coloneqq\overline{\alpha}_f(x)=\underline{\alpha}_f(x).
	\]
	By \cref{lemsingwilldef}, if $(X,f,x)$ dominates $(Y,g,y)$,
	then $\overline{\alpha}_f(x)\geq \overline{\alpha}_g(y)$
	and $\underline{\alpha}_f(x)\geq\underline{\alpha}_g(y)$.
\end{setup}

Applying Inequality~\eqref{equationdomineq1} of \cref{lemsingwilldef} to the case
where $Y=X$ and $M=L$, we get the following trivial upper bound:
let $f\colon X\dashrightarrow X$ be a dominant rational self-map,
$L$ any ample line bundle on $X$ and $h_L$ a Weil height function associated to $L$;
then there is a constant $C\geq 1$ such that for every $x\in X\setminus I(f)$, we have
\begin{equation}\label{equationtrivialupper}
	h_L^+(f(x))\leq Ch_L^+(x).
\end{equation}
For a subset $A\subseteq [1,\infty)$, define $A^{1/\ell}\coloneqq \{a^{1/\ell}\mid a\in A\}$.
We have the following simple properties,
where the second half of \ref{eq:alpha_pow} used Inequality~\eqref{equationtrivialupper}.

\begin{enumerate}[label=(\roman*),ref=(\roman*)]
	\item $A_f(x)\subseteq [1,\infty)$.
	\item $A_f(x)=A_f(f^{\ell}(x))$, for any $\ell\geq 0$.
	\item \label{eq:alpha_pow}
	      $A_{f}(x)=\bigcup_{i=0}^{\ell-1}(A_{f^{\ell}}(f^i(x)))^{1/\ell}$.
	      In particular, $\overline{\alpha}_{f^{\ell}}(x)=\overline{\alpha}_{f}(x)^{\ell}$,
	      $\underline{\alpha}_{f^{\ell}}(x)=\underline{\alpha}_{f}(x)^{\ell}$.
\end{enumerate}

\Cref{lem_subvar} below is easy but fundamental for the reduction to invariant subvarieties.

\begin{lemma}[cf.~e.g.,~{\cite[Lemma~2.5]{MMSZ20}}]\label{lem_subvar}
	Let $f\colon X\to X$ be a surjective endomorphism of a projective variety $X$
	and $W\subseteq X$ an $f$-invariant closed subvariety.
	Then $\alpha_{f|_W}(x)=\alpha_f(x)$ for any $x\in W(\overline\Q)$.
\end{lemma}

The next result generalises \cite[Theorem~1.4]{Mat20} to the singular case.

\begin{proposition}\label{proupboundarth}
	For every admissible triple $(X,f,x_0)$,
	we have\linebreak $\overline{\alpha}(x_0)\leq d_1(f)$.
\end{proposition}

\begin{proof}
	We may assume that $X$ is projective.
	Let $L$ be an ample divisor on $X$, and
	$h_L$ a Weil height function associated to $L$.
	We may assume that $h_L\geq 1$.
	After replacing $f$ by a suitable iteration and $x_0$ by $f^n(x_0)$ for some $n\geq 0$
	and noting that $d_1(f^n)=d_1(f)^n$ and by \ref{subsec_arithdeg}~\ref{eq:alpha_pow},
	we may assume that the Zariski closure $Z_f(x_0)$ of $O_f(x_0)$ is irreducible.
	After replacing $X$ by $Z_f(x_0)$ and noting that, by \cref{p:dyn_sub_var},
	$d_1(f|_{Z_f(x_0)})\leq d_1(f)$ while the value $\overline{\alpha}(x_0)$
	for the point $x_0$ being in $X$ or in $Z_f(x_0)$ is the same (cf.~\cref{lem_subvar}),
	we may assume that $O_f(x_0)$ is Zariski dense in $X$.

	Take a smooth projective variety $Y$ with a birational surjective morphism $\pi\colon Y\to X$.
	Take a Zariski closed proper subset $\widehat{Z}\subseteq X$
	such that $\pi$ restricts to an isomorphism $Y\setminus Z\to X\setminus\widehat{Z}\eqqcolon U$
	where $Z=\pi^{-1}(\widehat{Z})$.
	Lift $f$ on $X$ to $g\coloneqq\pi^{-1}\circ f\circ\pi\colon Y\dashrightarrow Y$.
	Let $H$ be an ample divisor on $Y$ and $h_H$ a Weil height function associated to $H$
	with $h_H\geq 1$.
	By \cref{lemsingwilldef}, there is a constant $B\geq 1$ such that for every $x\in U$,
	\begin{equation}\label{equationcomparelh}
		B^{-1}h_H(\pi^{-1}(x))\leq h_L(x)\leq Bh_H(\pi^{-1}(x)).
	\end{equation}

	The proof of \cite[Theorem 3.2]{Mat20} showed that for every $r>0$,
	there is a constant $K\geq 1$ and an integer $\ell\geq 1$, such that for every $y\in Y$,
	satisfying $y,g^{\ell}(y),\dots,g^{\ell n}(y)\in Y\setminus I(g^{\ell})$,
	we have
	\begin{equation}\label{equationmat}
		h_H(g^{\ell n}(y)) \le K(d_1(g)+r)^{\ell n}h_H(y).
	\end{equation}

	Set $V\coloneqq\pi(\pi^{-1}(U)\setminus I(g^{\ell}))$ and $Z'\coloneqq X\setminus V$.
	By Inequalities~\eqref{equationcomparelh} and \eqref{equationmat},
	for every point $p\in X$ satisfying $p,f^{\ell}(p),\dots,f^{\ell n}(p)\in V$, we have
	\begin{equation}\label{equgoodlocus}
		h_L(f^{\ell n}(p))\leq B^2K(d_1(f)+r)^{\ell n}h_L(p).
	\end{equation}

	By Inequality~\eqref{equationtrivialupper},
	there is a $C\geq (d_1(f)+r)^{\ell}$ such that for any $p\in X\setminus I(f^{\ell})$,
	we have
	\begin{equation}\label{equationtrivialupperl}
		h_L(f^{\ell}(p))\leq Ch_L(p).
	\end{equation}

	For every $n\geq 0$, define $W(n)\coloneqq\{0\leq i\leq n\mid f^{\ell i}(x_0)\in Z'\}$
	and $w_n\coloneqq\#W(n)$.
	Since $O_f(x_0)$ is Zariski dense in $X$,
	by the weak dynamical Mordell--Lang \cite[Theorem~1.10]{BHS20}
	(see also \cite[Theorem~1.4]{BGT15}, \cite[Theorem~1.6]{Gig14}),
	we have
	\begin{equation}\label{equationwdml}
		\lim_{n\to\infty}w_n/n=0.
	\end{equation}

	By Equality~\eqref{equationwdml} and \cref{lemupperboundmix} below, we have
	\begin{align*}
		\overline{\alpha}_{f}(x_0)^{\ell} & =\overline{\alpha}_{f^{\ell}}(x_0)=\limsup_{n\to \infty}h_{L}(f^{\ell n}(x_0))^{1/n}             \\
		                                  & \leq\lim_{n\to \infty} (B^2K)^{(w_n+1)/n}C^{2w_n/n}(d_1(f)+r)^{\ell(1-w_n/n)}=(d_1(f)+r)^{\ell}.
	\end{align*}
	So we have $\overline{\alpha}_{f}(x_0)\leq d_1(f)+r$.
	Letting $r$ tend to $0$, we get $\overline{\alpha}_{f}(x_0)\leq d_1(f)$.
	This proves \cref{proupboundarth}, modulo \cref{lemupperboundmix} below.
\end{proof}

\begin{lemma}\label{lemupperboundmix}
	With the assumption in the proof of \cref{proupboundarth},
	for $n\geq 0$, we have
	\[
		h_{L}(f^{\ell n}(x_0))\leq (B^2K)^{w_n+1}C^{2w_n} (d_1(f)+r)^{\ell(n-w_n)}h_L(x_0).
	\]
\end{lemma}

\begin{proof}
	Consider the decomposition of the finite set \(W(n)\)
	as a disjoint union of subsets of consecutive integers:
	\[
		W(n)=\bigsqcup_{i=1}^m\{n_i,n_i+1,\dots,n_i+s_i-1\}
	\]
	where $s_i\geq 1$, $n_{i+1}\geq n_i+s_i+1$ for $i=1,\dots,m$.
	We have $\sum_{i=1}^ms_i=w_n$.
	In particular $m\leq w_n$.
	Note that $\{0,\dots,n\}\setminus W(n)$ is a union of at most $m+1$
	maximal subsets of consecutive numbers.
	Applying Inequality~\eqref{equgoodlocus} for those maximal subsets of consecutive numbers
	in $\{0,\dots,n\}\setminus W(n)$ and Inequality~\eqref{equationtrivialupperl} for the others,
	we get
	\begin{align*}
		h_{L}(f^{\ell n}(x_0)) & \leq(B^2K)^{m+1}C^{\sum_{i=1}^m(s_i+1)}(d_1(f)+r)^{\ell(n-\sum_{i=1}^m(s_i+1))}h_L(x_0) \\
		                       & =(B^2K)^{m+1}C^{m+w_n}(d_1(f)+r)^{\ell(n-m-w_n)}h_L(x_0)                                \\
		                       & \leq(B^2K)^{w_n+1}C^{2w_n}(d_1(f)+r)^{\ell(n-w_n)}h_L(x_0).
	\end{align*}
	For the last inequality,
	we used the fact that \(1 \leq m \leq w_n\) and \(d_1(f)+r > 1\).
	This proves \cref{lemupperboundmix} (and also \cref{proupboundarth}).
\end{proof}

The property of KSC~\ref{Conj:KSC} is preserved under generically finite dominant morphism.
Indeed:

\begin{proposition}\label{p:kscequ}
	Let $\pi\colon X\to Y$ be a generically finite dominant morphism
	between quasi-projective varieties over $\overline{\Q}$.
	Let $f\colon X\to X$ and $g\colon Y\to Y$ be
	dominant endomorphisms
	satisfying $\pi\circ f=g\circ \pi$.
	Then we have:
	\begin{enumerate}[leftmargin=1.8em]
		\item The $f$-orbit of $x\in X$ is Zariski dense if and only if
		      so does the $g$-orbit of $\pi(x)\in Y$.
		      In this case, $\alpha_f(x)$ exists if and only if
		      so does $\alpha_g(\pi(x))$ and they take the same value.
		\item KSC~\ref{Conj:KSC} holds for $(X,f)$ if and only if KSC~\ref{Conj:KSC} holds for $(Y,g)$.
	\end{enumerate}
\end{proposition}

\begin{proof}
	We first show that (1) implies (2).
	If KSC~\ref{Conj:KSC} holds for $(Y,g)$, then KSC~\ref{Conj:KSC} holds for $(X,f)$ by (1).
	Assume that KSC~\ref{Conj:KSC} holds for $(X,f)$. Let $y$ be a point in $Y(\overline{\Q})$ of Zariski dense orbit.
	Since $\pi(X)$ contains a non-empty open subset of $Y$, there is $m\geq 0$ such that $g^m(y)\in \pi(X)$. Pick $x\in X(\overline{\Q})$ with $\pi(x)=y.$
	Since $d_1(f)=d_1(g)$ (cf.~e.g.,~\cite{Dan20}), by (1) we get $$\alpha_g(y)=\alpha_g(g^m(y))=\alpha_f(x)=d_1(f)=d_1(g).$$ Hence KSC~\ref{Conj:KSC} holds for $(Y,g)$.

	The first statement of (1) is clear, so we only need to prove the second statement.
	Pick projective compactifications $X'$ of $X$ and $Y'$ of $Y$
	such that $\pi$ extends to a morphism $\pi'\colon X'\to Y'$.
	Pick ample line bundle $M, L$ on $X', Y'$ respectively, such that
	$\pi'^*L-M=\mathcal{O}_{X'}(E)$ for some effective divisor $E$ on $X'$.
	Pick a non-empty open subset $U_Y$ on $Y\setminus \pi'(E)$ such that $\pi: U_X:=\pi^{-1}(U_Y)\to U_Y$ is finite.
	Pick Weil heights $h_M, h_L$ on $X', Y'$ respectively such that
	\begin{enumerate}[leftmargin=1.8em]
		\item[(i)] $h_M\geq 1, h_L\geq 1$;
		\item[(ii)] $h_M\leq \pi'^*h_L$ on $U_X$;
		\item[(iii)] $\pi'^*h_L\leq Ch_M$ for some $C>0$.
	\end{enumerate}
	Set $y:=\pi(x).$
	Set $G:=\{n\geq 0|\,\, g^n(y)\in U_Y\}$ and $B:=\Z_{\geq 0}\setminus G$.
	Since $O_g(y)$ is Zariski dense, $G$ is infinite.
	For $n\in  G$, we have
	\begin{equation}\label{equninG}
		h_M(f^n(x))\leq h_L(g^n(y))\leq Ch_M(f^n(x)).
	\end{equation}
	By (\ref{equationtrivialupper}),  there is $D>0$ such that
	for every $z\in X$, $w\in Y$, we have
	\begin{equation}
		h_M(f(z))\leq Dh_M(z) \text{ and } h_L(g(z))\leq Dh_L(z).
	\end{equation}

	Set $Z:= Y\setminus U_Y$. For every $r\geq 0$, set
	$$Z_r:=\cap_{i=0}^r g^{-i}(Z),$$
	which is a decreasing sequence of Zariski closed subsets of $Y$.
	Then there is $R\geq 0$ such that $Z_r=Z_R$ for $r\geq R.$
	For every $n\geq B\cap \Z_{>R}$,  there are $r(n),s(n)\in \{1,\dots, R+1\}$ such that $n-r(n), n+s(n)\in G.$
	Then we get
	\begin{equation}\label{equninBg}
		D^{-R-1}h_L(g^{n+s(n)}(y))\leq h_L(g^n(y))\leq D^{R+1}h_L(g^{n-r(n)}(y))
	\end{equation}
	and
	\begin{equation}\label{equninBf}
		D^{-R-1}h_M(f^{n+s(n)}(y))\leq h_M(f^n(y))\leq D^{R+1}h_M(f^{n-r(n)}(y)).
	\end{equation}
	We then conclude the proof by (\ref{equninG}), (\ref{equninBg}) and (\ref{equninBf}).
\end{proof}

\Cref{l:iter_inv} below allows us to replace $f$ by any positive power whenever needed.

\begin{lemma}\label{l:iter_inv}
	Let $X$ be a (quasi-projective,
	when we discuss KSC~\ref{Conj:KSC}) variety
	and $f\colon X\dashrightarrow X$ a dominant rational self-map.
	The following statements are equivalent.
	\begin{enumerate}[leftmargin=2em]
		\item $(X,f)$ satisfies ZDO~\ref{Conj:ZDO} (resp.~AZO~\ref{Conj:AZO}, KSC~\ref{Conj:KSC}).
		\item There is an $\ell\geq 1$,
		      such that $(X,f^{\ell})$ satisfies ZDO~\ref{Conj:ZDO}
		      (resp.~AZO~\ref{Conj:AZO}, KSC~\ref{Conj:KSC}).
	\end{enumerate}
\end{lemma}

\begin{proof}
	For the ZDO and AZO parts, we refer to \cite[Propositions 2.2 and 3.29]{Xie19}.

	For the KSC part, note that $d_1(f)^{\ell}=d_1(f^{\ell})$
	and $\alpha_f(x)^{\ell}=\alpha_{f^{\ell}}(x)$ when the latter exists,
	for any $\ell\geq 1$ and $x\in X$ (cf.~\ref{subsec_arithdeg}~\ref{eq:alpha_pow}).
	Also, $O_f(x)$ is Zariski dense if and only if so is $O_{f^{\ell}}(x)$.
	Then the lemma follows.
\end{proof}

\section{Endomorphisms descending along a fibration}

This section treats endomorphisms descending along fibrations,
especially $\A^1$-fibrations.

\begin{lemma}\label{l:surj}
	Let $X_i$, $B_i$ be varieties,
	and $\pi_i\colon X_i\to B_i$ a surjective morphism with a general fibre irreducible ($i=1,2$).
	Let $f\colon X_1\to X_2$ be a finite surjective morphism
	which descends to a surjective morphism $g\colon B_1\to B_2$.
	Then $f|_{(X_1)_b}\colon (X_1)_b\to (X_2)_{g(b)}$ is surjective for any closed point $b\in B_1$.
\end{lemma}

\begin{proof}
	Set $X_2'\coloneqq X_2\times_{B_2}B_1$.
	Since a general fibre of $\pi_2$ is geometrically irreducible, $X_2'$ is irreducible.
	Denote by $\phi_1\colon X_2'\to X_2$
	and $\phi_2\colon X_2'\to B_1$ the two natural projections, respectively.
	The natural morphism $\tau\colon X_1\to X_1\times_{B_2}B_1$ is a closed embedding,
	noting that $g$ is separated.
	Since $f$ is finite, $f\times_{B_2}\id\colon X_1\times_{B_2}B_1\to X_2'=X_2\times_{B_2}B_1$ is finite.
	Then $f'\coloneqq(f\times_{B_2}\id)\circ\tau\colon X_1\to X_2'$ is finite.
	We have the following commutative diagram.
	\[
		\xymatrixcolsep{3pc}
		\xymatrix{
		X_1\ar[r]^{f'}\ar[d]_{\pi_1}\ar@/^16.5pt/[rr]^{f}	& X_2'\ar[r]^{\phi_1}\ar[dl]^{\phi_2}	& X_2 \ar[d]^{\pi_2}\\
		B_1\ar[rr]^g										&										& B_2
		}
	\]
	Since $\phi_1$ is generically finite and $f$ is surjective, $f'$ is dominant.
	Then $f'$ being finite and dominant implies it is surjective.
	For any $b\in B_1$, $f'((X_1)_b)=(X_2')_b$.
	Since $\phi_1((X_2')_b)=(X_2)_{g(b)}$,
	we have $f((X_1)_b)=\phi_1(f'((X_1)_b))=\phi_1((X_2')_b)=(X_2)_{g(b)}$.
\end{proof}

\begin{lemma}\label{l:descending}
	Let $f\colon X_1\to X_2$ be a morphism between normal varieties,
	and $\pi_i\colon X_i\to B_i$ a surjective morphism to a smooth curve ($i=1,2$).
	Suppose that a general fibre $F_1$ of $\pi_1$ is irreducible.
	Suppose further that $f(F_1)$ is contained in a fibre of $\pi_2$
	(this holds when $F_1$ is a curve with $\overline{\kappa}(B_2)>\overline{\kappa}(F_1)$,
	e.g., when $\pi_1$ is an $\A^1$-fibration and $\overline{\kappa}(B_2)\geq 0$).
	Then $f$ descends to a morphism $f|_{B_1}\colon B_1\to B_2$ (which is surjective if so is $f$).
\end{lemma}

\begin{proof}
	Take normal compactifications $X_i\subseteq\overline{X_i}$, $B_i\subseteq\overline{B_i}$
	such that the natural extensions
	$\overline{\pi_i}\colon\overline{X_i}\dashrightarrow\overline{B_i}$ ($i=1,2$)
	and $\overline{f}\colon\overline{X_1}\dashrightarrow\overline{X_2}$ are morphisms.

	Since $\overline{B_1}$ is smooth of dimension one,
	$\overline{\pi_1}$ is flat and hence has equi-dimensional connected fibres.
	By the assumption and the rigidity lemma (cf.~\cite[Lemma~1.15]{Deb01}),
	$\overline{\pi_2}\circ\overline{f}$ contracts one and hence every fibre of $\overline{\pi_1}$,
	and then $\overline{\pi_2}\circ\overline{f}$ factors through $\overline{\pi_1}$
	via some $\overline{f}|_{\overline{B_1}}\colon\overline{B_1}\to\overline{B_2}$.
	Since $\pi_1\colon X_1\to B_1$ is surjective,
	every $b\in B_1$ is mapped to $\pi_2(f((X_1)_b))\in B_2$,
	i.e., $\overline{f}|_{\overline{B_1}}(B_1)\subseteq B_2$.
	Hence the (surjective) morphism $f\colon X_1\to X_2$ descends,
	via the surjections $\pi_i\colon X_i\to B_i$,
	to a (surjective) morphism $f|_{B_1}\colon B_1\to B_2$.

	Note that if $F_1$ is a curve and $f(F_1)$ dominates $B_2$ via $\pi_2$
	then $\overline{\kappa}(F_1)\geq\overline{\kappa}(f(F_1))\geq\overline{\kappa}(B_2)$,
	and note also $\overline{\kappa}(\A^1)=-\infty$.
	Then the lemma follows.
\end{proof}

\begin{proposition}\label{p:bad_fib}
	Let $\pi\colon X\to B$ be a surjective morphism from a normal quasi-projective variety to a smooth curve with a general fibre irreducible.
	Let $f\colon X\to X$ be a finite surjective morphism which descends to a surjective morphism $g\colon B\to B$
	(this is the case when \cref{l:descending} is satisfied).
	Let $\Sigma_0\subseteq B$ (resp.~$\Sigma_1\subseteq B$) be the set of $b\in B$ such that the fibre $X_b$ is reducible (resp.~irreducible and non-reduced).
	Then we have:
	\begin{enumerate}[leftmargin=2em]
		\item $g^{-1}(\Sigma_0)=\Sigma_0$.
		      Hence $g$ induces a surjective endomorphism of $B\setminus\Sigma_0$.
		\item If $g$ is {\'e}tale (this is the case when $\overline{\kappa}(B)\geq 0$; cf.~\cref{lem:k>0_et}), then $g^{-1}(\Sigma_1)=\Sigma_1$, hence $g$ induces a surjective endomorphism of\linebreak $B\setminus (\Sigma_0\cup\Sigma_1)$.
		\item If $\overline{\kappa}(B)=1$, or if $\overline{\kappa}(B)\geq 0$ and $\pi$ has a reducible fibre or a non-reduced fibre, then $g$ is an automorphism of finite order.
	\end{enumerate}
\end{proposition}

\begin{proof}
	\Cref{l:surj} implies that $f$ maps the fibre $X_b$ onto $X_{g(b)}$ for any $b\in B$.
	Thus $g^{-1}(\Sigma_0)\subseteq\Sigma_0$.
	Now $\Sigma_0$ is a finite set, so $g^{-1}(\Sigma_0)=\Sigma_0$.

	If $g$ is {\'e}tale, then $g^*b$ is reduced for any $b\in B$.
	This implies that $g^{-1}(\Sigma_1)\subseteq\Sigma_1$.
	Since $\Sigma_1$ is a finite set, we have $g^{-1}(\Sigma_1)=\Sigma_1$.

	(3) follows from (1), (2) and \cref{l:finord}.
\end{proof}

\begin{lemma}\label{lem:A1_fib}
	Let $X$ be a smooth affine surface and $\pi\colon X \to B$ an $\A^1$-fibration.
	Then every fibre has support equal to a disjoint union of $\A^1$'s.
	Hence, either $\pi$ is an $\A^1$-bundle, or it has a reducible or non-reduced fibre.
\end{lemma}

\begin{proof}
	Extend $\pi$ to $\overline{\pi}\colon\overline{X}\to \overline{B}$ from a smooth projective surface to a smooth projective curve with the boundary $D=\overline{X}\setminus X$ an SNC divisor.
	Then $D$ consists of a cross-section of $\pi$ and some fibre components.
	Noting that $D$ supports a connected ample effective divisor by \cref{l:D_amp},
	the affine $X$ contains no compact $\PP^1$,
	and every fibre of $\overline{\pi}\colon\overline{X}\to \overline{B}$ consists of $\PP^1$'s
	and has dual graph a connected tree, the lemma follows.
\end{proof}

\begin{proposition}\label{p:big_pi_1}
	Let $X$ be a smooth affine surface over $\kk$ and $f\colon X\to X$ a finite surjective morphism.
	Let $\pi\colon X\to B$ be a surjective $\A^1$-fibration.
	Suppose that $\vert\pi_1^{\et}(X)\vert=\infty$, and $B=\A^1$ or $\PP^1$.
	Then $f$ descends to an automorphism $f|_B$ on $B$ of finite order.
\end{proposition}

\begin{proof}
	Our proof closely follows \cite[Lemmas~3.5--3.6]{GZ08}.
	We consider the case $B=\A^1$ only since the case $B=\PP^1$ is similar.

	We first prove that the surjective endomorphism $f$ descends to a surjective endomorphism $f|_B$ on $B$.
	By \cref{l:descending}, we only need to show the property $P(\kk)$: $f(F)$ is contained in a fibre of $\pi$ for a general fibre $F$ of $\pi$.
	The property $P(\kk)$ satisfies \eqref{eq:property} of \cref{r:red_2C}.
	So we may assume $\kk=\C$.
	Note that the (topological) fundamental group $\pi_1(X)$ of $X$ is infinite by the assumption that $\vert\pi_1^{\et}(X)\vert=\infty$.

	Let $F_1,\dots,F_r$ be all multiple fibres of $\pi\colon X\to B$ with multiplicity\linebreak $m_i\geq 2$ (cf.~\cref{d:mult_fib}).
	If $r\leq 1$, then $\pi_1(X)$ is finite cyclic of order upper-bounded by the multiplicity of $F_r$
	(cf.~Proof of \cite[Lemma~1.5]{Nor83}, or \cite[Lemma~1.1.11]{GMM21}).
	Thus $r\geq 2$.

	By \cref{p:fat_fib}, there is a finite surjective morphism $B'\to B$ with ramification index equal to $m_i$ at every point of $B'$ lying over the point $\pi(F_i)$ such that the normalisation $X'$ of $X \times_B B'$ is {\'e}tale over $X$, the induced fibration $\pi'\colon X' \to B'$ has no multiple fibre and the fibres of $\pi'$ lying over $F_i$ are all reducible (see also \cref{lem:A1_fib}).

	Suppose the contrary that the property $P(\C)$ is false.
	Then the normalisation of $f(F)$ equals $\A^1$ (because $\overline{\kappa}(f(F)) \leq \overline{\kappa}(F)=\overline{\kappa}(\A^1)=-\infty$), so it is simply connected and dominates $B$ via $\pi\colon X\to B$.
	Take an irreducible component $F'\subseteq X'$ of the inverse of $f(F)$ via the {\'e}tale map $X'\to X$.
	Then $F'$ has normalisation equal to $\A^1$ and dominates $B'$.
	Thus $-\infty=\overline{\kappa}(\A^1)\geq\overline{\kappa}(\pi'(F'))\geq\overline{\kappa}(B')$.
	So $B'\cong\A^1$.
	This and that $\pi'\colon X'\to B'$ has all fibres non-multiple, with the help of \cite[Lemma~1.5]{Nor83}, imply that $\pi_1(X')=(1)$.
	Hence $\pi_1(X)$ is finite too, a contradiction to the assumption.
	Thus $P(\C)$ and hence $P(\kk)$ hold true.
	So $f$ descends to a surjective endomorphism $f|_B$.

	Next we show that $f|_B$ is an automorphism of finite order.
	Clearly, this property also satisfies \eqref{eq:property} of \cref{r:red_2C}.
	So we may assume that $\kk=\C$ in the following.

	Let $\Sigma_0\subseteq X$ (resp.~$\Sigma_0'\subseteq X'$) be the set of points over which the fibres of $\pi$ (resp.~$\pi'$) are reducible.
	By \cref{p:bad_fib}, $(f|_B)^{-1}(\Sigma_0)=\Sigma_0$.
	Iterating $f$, we may assume that $f^{-1}$ stabilises every component in every reducible fibre of $\pi\colon X\to B$.
	Thus $f|_B$ restricts to a surjective endomorphism of $B\setminus\Sigma_0$.

	The argument in \cite[Lemma 3.5]{GZ08} shows that $f$ induces an isomorphism of the group $\pi_1(X)$.
	Thus the covering theory implies that $f$ lifts to a finite surjective morphism $f'\colon X'\to X'$, via the {\'e}tale covering $X'\to X$.

	By the compatibility of $f,\pi$ on $X$ and $f',\pi'$ on $X'$, our $f'$ descends to a surjective morphism $f'|_{B'}\colon B'\to B'$.
	The same \cref{p:bad_fib} implies that $(f'|_{B'})^{-1}(\Sigma_0')=\Sigma_0'$
	and $f'|_{B'}$ restricts to a surjective endomorphism of $B'\setminus\Sigma_0'$.

	Note that $\vert\Sigma_0'\vert\geq 2$ (since the fibres lying over the $r\geq 2$ of fibres $F_i$ are reducible).
	Thus $B'\setminus\Sigma_0'$ is the affine curve $B'$ with at least two points removed, hence $\overline{\kappa}(B'\setminus\Sigma_0')=1$
	(alternatively, since all fibres of the $\A^1$-fibration $\pi'\colon X'\to B'$ are non-multiple,
	$\vert\pi_1(X')\vert=\infty$ and \cite[Lemma~1.5]{Nor83} imply that $\vert\pi_1(B')\vert=\infty$,
	so $\overline{\kappa}(B')\geq 0$ and hence $\overline{\kappa}(B'\setminus\Sigma_0')=1$).
	Therefore, $f'|_{B'\setminus\Sigma_0'}$ and hence $f|_B$ are automorphisms of finite order (cf.~\cref{l:finord}).
\end{proof}

\begin{lemma}\label{lem:cptf}
	Let $\pi_i \colon X_i \to B_i$ be an $\A^1$-bundle from a smooth surface to a smooth curve ($i=1,2$).
	Let $f\colon X_1 \to X_2$ be a finite surjective morphism
	which descends to a surjective morphism $g \colon B_1 \to B_2$.
	Take any $\PP^1$-bundle $\pi_i' \colon V_i \to B_i$ as a partial compactification of $\pi_i$ (which exists by \cref{l:exc}).
	Then $f$ extends to a surjective morphism $f \colon V_1 \to V_2$.
\end{lemma}

\begin{proof}
	Note that $H_i=V_i \setminus X_i$ is an irreducible $\pi_i'$-horizontal curve
	since every fibre of $\pi_i \colon X_i \to B_i$ is $\A^1$.
	Extend $f$ as a rational map $f: V_1 \dasharrow V_2$ with indeterminacy being several points on $H_1$.
	Take a composition of blow ups $\mu \colon V_1' \to V_1$
	such that $f'=f \circ \mu$ is a morphism and $\mu(\Exc(\mu)) \subseteq H_1$.
	Since $f$ descends to $g$, we have $\pi_2 \circ f' = g \circ \pi_1 \circ \mu$.
	Take a $(-1)$-curve $C \subseteq \Exc(\mu)$.
	If $f'(C)$ is not a point, then $f'(C)=H_2$, since $f^{-1}(X_2)=X_1$;
	but then $\pi_2(f'(C))=g(\pi_1(\mu(C)))$ is a point, a contradiction.
	So we may contract $C$, preserving $f'$ as a morphism.
	Continuing this process, $f \colon V_1 \to V_2$ is a (surjective) morphism.
\end{proof}

\begin{lemma}\label{lem:ell_base}
	Let $X$ be a smooth affine surface and $f\colon X\to X$ a finite surjective morphism
	with $\deg(f)\geq 2$.
	Let $\pi\colon X\to B$ be a surjective $\A^1$-fibration.
	Suppose $B$ has an elliptic curve $\overline{B}$ as its compactification.
	Then $(X,f)$ satisfies both AZO~\ref{Conj:AZO} and KSC~\ref{Conj:KSC}.
\end{lemma}

\begin{proof}
	Note that $\overline{\kappa}(B)\geq\overline{\kappa}(\overline{B})=0$.
	By \cref{l:descending}, $f$ descends to a surjective endomorphism $g\colon B\to B$.
	If $B\neq\overline{B}$, then $\overline{\kappa}(B)=1$ and $g$ is of finite order by \cref{l:finord};
	hence, both AZO~\ref{Conj:AZO} and KSC~\ref{Conj:KSC} hold trivially
	(cf.~\cref{l_inv_fun_field,l:iter_inv}).
	Assume $B=\overline{B}$.
	We may assume also that $\pi\colon X\to B$ is an $\A^1$-bundle; otherwise,
	$g$ is of finite order by \cref{lem:A1_fib} and \cref{p:bad_fib}, and we are done again.
	By \cref{lem:cptf}, $X$ embeds into a $\PP^1$-bundle $V$ over $B$
	and $f$ extends to a surjective endomorphism on $V$ with $f(X) = X$.
	The lemma follows from \cref{r:AZO_KSC_surf}.
\end{proof}

\section{Semiabelian varieties: Proof of Theorem~\ref{Thm_semiAb}}

The aim of this section is to prove \cref{Thm_semiAb} the proof of which is similar to the proof of \cite[Theorem 1.14]{Xie19}.
In Subsection~\ref{setup:AdelicTopology} and \cref{rem:intersection_of_finite_adelic_open},
we have briefly recalled some basic properties of the adelic topology and given basic examples of adelic open subsets.
See \cite[Section 3]{Xie19} for the definition and more detailed discussions of adelic topology.
We begin with some notation and lemmas.

\begin{lemma}\label{lemfixdimone}
	Suppose that there is a subvariety $V$ of $G$ such that $\dim V\geq 1$ and $f|_V=\id$.
	Then $\kk(G)^f\neq \kk$.
\end{lemma}

\begin{proof}
	We may assume that $0\in V$.
	Then $f$ is an isogeny.
	We have $V\subseteq\Ker(f-\id)$.
	So $\dim\Ker(f-\id)\geq 1$.
	Write the minimal polynomial of $f$, killing $f$ on $G$ (or equivalently $f_{\C}$ on $G_{\C}$), as $(1-t)^rP(t)$ where $P(1)\neq 0$.
	We have $r\geq 1$.
	Set $N\coloneqq(\id-f)^{r-1}P(f)$ and $B=N(G)$.
	Then $\dim B\geq 1$ and $B\subseteq\Ker(f-\id)$.
	Further, $f$ descends to $f|_B=\id\colon B\to B$ via $N$.
	Pick a nonconstant rational function $F$ on $B$.
	Set $H\coloneqq N^*F=F\circ N$, which is a nonconstant rational function on $G$.
	We conclude the proof with:
	\[
		f^*H=F\circ N\circ f=F\circ(f|_B)\circ N=F\circ \id \circ N=H.\qedhere
	\]
\end{proof}

For every closed subset $V \subseteq G$, define
\[
	S_V\coloneqq \{a\in G\mid a+V= V\}.
\]
Then $S_V$ is a group subvariety of $G$.
Denote by $S_V^0$ the identity component of $S_V$.
Then $S_V^0\subseteq G$ is a semiabelian subvariety.

Assume that $V$ is irreducible and invariant under $f$.
Set $B\coloneqq G/S_V^{0}$ and denote by $\pi\colon G\to B$ the quotient morphism.
There is a unique endomorphism $f|_B\colon B\to B$ such that $f|_B\circ \pi=\pi\circ f$.
Since $f$ is dominant, $f|_B$ is also dominant.
Set $V_B\coloneqq\pi(V)$.
By \cite[Theorem 3]{Abr94}, $\overline{\kappa}(V_B)=\dim V_B$.
By \cref{l:finord}, there is some $\ell>0$ such that $(f^{\ell}|_B)|_{V_B}=\id$.
Observe that if $\dim V_B=0$, $V$ takes the form $a+S_V^0$ for some $a\in V$.

\begin{lemma}\label{lemwbpodim}
	Assume that $\dim V_B\geq 1$.
	Then $\kk(G)^f\neq\kk$.
\end{lemma}

\begin{proof}
	By \cref{l_inv_fun_field}, replacing $f$ by $f^{\ell}$, we may assume that $(f|_B)|_{V_B}=\id$.
	Since $\pi$ is surjective,
	we only need to show that there is a nonconstant rational function $H$ on $B$
	satisfying $(f|_B)^*H=H$.
	This is achieved by applying \cref{lemfixdimone} to $f|_B$, $B$ and $V_B$.
\end{proof}

\begin{lemma}\label{lemnoinvrtransub}
	Suppose that $\kk(G)^f=\kk$.
	Then:
	\begin{enumerate}[leftmargin=2em]
		\item Every irreducible $f$-invariant subvariety $V$ takes the form $a+G_0$
		      where $a\in G$ and $G_0$ is a semiabelian subvariety of $G$.
		\item Suppose further $f$ is an isogeny.
		      Then the fixed point set $\Fix(f)$ of $f$ is finite and $V\cap\Fix(f)\neq\emptyset$.
	\end{enumerate}
\end{lemma}

\begin{proof}
	By \cref{lemwbpodim}, $V = a+G_0$, with $a\in G$ and $G_0$ a semiabelian subvariety of $G$.

	Now assume that $f$ is an isogeny.
	Since $f(V)=V$, $f(a)+f(G_0)=a+G_0$.
	It follows that $f(G_0)=G_0$ and $f(a)-a\in G_0$.
	By \cref{lemfixdimone}, $\Fix(f)$ is finite and $(f-\id)|_{G_0}$ is an isogeny.
	So there is some $x\in G_0$ such that $f(x)-x=a-f(a)$.
	It follows that $f(a+x)=a+x$ and $a+x\in a+G_0=V$.
	This concludes the proof.
\end{proof}

Now we are ready for:

\begin{proof}[Proof of \cref{Thm_semiAb}]
	We may assume that $\kk(G)^f=\kk$.
	Also, we may replace $f$ by an iteration (cf.~\cref{l_inv_fun_field,l:iter_inv}).
	Let $K$ be a subfield of $\kk$ which is finitely generated over $\Q$ satisfying $\overline{K}=\kk$,
	such that $G$ and $f$ are defined over $K$.
	There exists a semiabelian variety $G_K$ over $K$ and an endomorphism $f_K\colon G_K\to G_K$ such that $G=G_K\times_{\Spec K}\Spec\kk$ and $f=f_K\times_{\Spec K}\id$.

	\medskip

	We first treat the case where $f$ is an isogeny.
	Denote by $G[2]$ the finite subgroup of the $2$-torsion points in $G$.
	By \cref{lemfixdimone}, $\Fix(f)$ is finite.
	After replacing $K$ by a finite extension,
	we may assume that all points in $\Fix(f)\cup G[2]$ are defined over $K$.

	By abusing notation,
	we will use addition to denote the group operation on the semiabelian variety $G$,
	i.e., regard $G$ as an additive group.
	For every $\ell\in\Z$, denote by $[\ell]\colon G\to G$ the morphism $x\mapsto\ell x$.
	Since $[3]\circ f=f\circ [3]$, we have $[3](\Fix(f))\subseteq\Fix(f)$.
	There is some $m\geq 1$ such that for every $x\in\Fix(f)$, we have $[3^{2m}](x)=[3^m](x)$.

	By \cite[Proposition 3.24]{Xie19}, replacing $f$ by a positive iteration,
	there is a nonempty adelic open subset $P$ of $G(\kk)$ such that for every point $z\in P$,
	the orbit closure $Z_f(z)\coloneqq\overline{O_f(z)}$ is irreducible.
	Then we have $f(Z_f(z))=Z_f(z)$.
	By \cref{lemnoinvrtransub},
	$Z_f(z)$ takes the form $a+H$ where $H$ is a semiablian subvariety of $G$ and $a\in\Fix(f)$.

	Our $G$ fits the following exact sequence
	\[
		1\longrightarrow\G_m^r\longrightarrow G\longrightarrow A\longrightarrow 1,
	\]
	where $A$ is an abelian variety of dimension $s$.
	We have $|G[2]|=2^{2s+r}$.
	Moreover, for every semiabelian subvariety $H'$ of $G$, we have
	\begin{equation*}\label{eqsemiabl}
		\tag{$\ast$}
		\vert G[2]\cap H'\vert=\vert H'[2]\vert\leq \vert G[2]\vert
	\end{equation*}
	and the last equality holds if and only if $H'=G$.

	Pick an embedding $\tau\colon K\hookrightarrow\C_3$.
	We note that $0\in G_K(\C_3)$ is an attracting fixed point for $[3]$.
	There exists an open neighbourhood $U\subseteq G_K(\C_3)$ of $0$ such that for every $x\in U$,
	$\lim_{n\to\infty}[3^n]x=0$.
	Let $P$ be an adelic open subset of $G(\kk)$.
	Then
	\[
		C\coloneqq P\cap (\cap_{y\in G[2]}G_K(\tau,y+U))
	\]
	is a nonempty basic adelic open subset of $G(\kk)$
	(cf.~\cref{rem:intersection_of_finite_adelic_open}).

	We only need to show that for every $x\in C$, $Z_f(x)=G$.
	Denote by $I_{\tau}$ the set of field embeddings $j\colon\kk\hookrightarrow\C_3$ with $j|_K=\tau$.
	For $j\in I_{\tau}$,
	denote by $\phi_j\colon G(\kk)\hookrightarrow G_K(\C_3)$
	the embedding induced by $j\colon \kk\hookrightarrow\C_3$.
	For every $y\in G[2]$, there is some $j_y\in I_{\tau}$ such that $a_y\coloneqq \phi_{j_y}(x)\in y+U$.
	As remarked early on,
	we can write $Z_f(x)=a+H$ with $a\in\Fix(f)$ and $H$ a semiablian subvariety of $G$.
	Then for $n\geq 1$, we have $[3^{nm}](Z_f(x))=[3^{nm}](a+H)=[3^m](a)+H$.

	We note that $a_y=\phi_{j_y}(x)\in \phi_{j_y}(Z_f(x))$ for every $y\in G[2]$.
	Then for every $n\geq 0$, $y\in G[2]$, we have
	\[
		y+[3^{nm}](a_y-y)=[3^{nm}](a_y)\in \phi_{j_y}([3^m](a)+H).
	\]
	Since $a_y-y\in U$, letting $n\to\infty$, we get $y\in\phi_{j_y}([3^m](a)+H)$.
	Since $y$ is defined over $K$, we have $y\in [3^m](a)+H$.
	It follows that $G[2]\subseteq [3^m](a)+H$.
	In particular, $0\in [3^m](a)+H$, so $H=[3^m](a)+H\supseteq G[2]$.
	Thus, $H[2]=G[2]$.
	Hence $H=G$ as we remarked after the display \eqref{eqsemiabl} early on.
	It follows that $Z_f(x)=G$.

	\medskip

	Now we treat the general case.
	Let $V$ be a subvariety of $G$ which has minimal dimension in all $f$-periodic subvarieties.
	By \cref{lemnoinvrtransub}, $V$ is a translate of a semiabelian subvariety of $G$.
	After changing the origin of $G$ and replacing $f$ by a suitable iterate,
	we may assume that $V$ itself is a semiabelian subvariety of $G$ and $f(V)=V$.
	Set $B\coloneqq G/V$ and denote by $\pi\colon G\to B$ the quotient morphism.
	There is an endomorphism $f|_B\colon B\to B$ such that $f|_B\circ \pi=\pi\circ f$.
	Since $f(V)=V$ and $f$ is dominant, $f|_B$ is an isogeny.
	Then, by the completed isogeny case,
	there is a nonempty adelic open subset $D$ of $B(\kk)$ such that for every $x\in D$,
	the orbit $O_{f|_B}(x)$ is Zariski dense in $B$.
	By \cite[Proposition 3.24]{Xie19}, after replacing $f$ by a positive iteration,
	there is a nonempty adelic open subset $P$ of $G(\kk)$ such that for every point $z\in P$,
	its orbit closure $Z_f(z)$ is irreducible.

	We claim that for every $x\in \pi^{-1}(D)\cap P$, we have $Z_f(x)=G$.
	Indeed, since $O_{f|_B}(\pi(x))$ is Zariski dense in $B$, we have $\pi(Z_f(x))=B$.
	So $Z_f(x)\cap V\neq \emptyset$.
	Since $f(Z_f(x)\cap V)\subseteq Z_f(x)\cap V$,
	there is an $f$-periodic subvariety contained in $Z_f(x)\cap V$.
	The minimality of $\dim V$ implies that $V\subseteq Z_f(x)$.
	By \cref{lemnoinvrtransub} and noting that $V$ is an semiabelian subvariety,
	$Z_f(x)$ is a semiabelian subvariety of $G$.
	Then $Z_f(x)$ contains a fibre $V$ of $\pi\colon G\to B$ and dominates $B$, hence is equal to $G$.
\end{proof}

\section[Case of log Kodaira dimension
  \texorpdfstring{$\geq 0$}{>= 0}:\\\mbox{} Proofs of Theorems~\ref{Thm_k>0_str} and \ref{Thm_k>0_conj}]{Case of log Kodaira dimension
  \texorpdfstring{$\geq 0$}{>= 0}: Proofs of Theorems~\ref{Thm_k>0_str} and \ref{Thm_k>0_conj}}

This section deals with smooth affine surfaces of non-negative log Kodaira dimensions.

\begin{proposition}\label{p:kappa=1_per}
	Let $X$ be a smooth affine surface
	and $f\colon X\to X$ a finite surjective morphism of degree $\geq 2$.
	Suppose $\overline{\kappa}(X)=1$.
	Then we have:
	\begin{enumerate}[leftmargin=2em]
		\item There is a surjective morphism $\pi\colon X\to B$ to a smooth curve $B$
		      such that $f$ descends to an automorphism $f|_B\colon B\to B$ of finite order.
		\item There exists a finite surjective morphism $B''\to B$ from a smooth curve
		      such that the normalisation $X''$ of $X \times_{B} B''$ is an {\'e}tale cover of $X$
		      and the induced morphism $\pi''\colon X''\to B''$ is a trivial $\G_m$-bundle.
		      Moreover, after iteration, $f$ lifts to an endomorphism $f''$ on $X''$
		      such that\linebreak $\pi''\circ f''=\pi''$.
	\end{enumerate}
\end{proposition}

\begin{proof}
	Take a log smooth compactification $(V,D)$ of $X$.
	Then $\kappa(V,K_V+D)=\overline{\kappa}(X)=1$.
	Hence $K_V+D$ has a Zariski decomposition $K_V+D=P+N$.
	By \cref{p:nef_free}, the nef part $P$ is semi-ample.

	Let $\Gamma$ be a desingularisation of the graph of the rational map
	$f\colon V_1=V\dashrightarrow V=V_2$ (extending $f\colon X\to X$)
	with two projections $p_i\colon\Gamma\to V_i$.
	Since $f^{-1}(X)=X$, we have $p_1^{-1}(D)=p_2^{-1}(D)\;(\eqqcolon D_{\Gamma}$).
	By the log ramification divisor formula,
	$K_{\Gamma}+D_{\Gamma}=p_i^*(K_V+D)+E_i$ for some effective divisor $E_i$ for $i=1,2$
	(cf.~\cite[Theorem 11.5]{Iit82}).
	Now the following composition self-map $f^*$ is an injective linear transformation
	\begin{equation*}\label{eqinjhom}
		\tag{$\dagger$}
		f^*\colon H^0(V_2,s(K_V+D))\to H^0(\Gamma,s(K_{\Gamma}+D_{\Gamma}))\cong H^0(V_1,s(K_V+D))
	\end{equation*}
	and hence an isomorphism, for any $s\geq 1$.

	Take $s$ sufficiently large and divisible.
	Then $\Phi_{|s(K_V+D)|}=\Phi_{|sP|}$ as rational maps,
	where the latter is a well-defined morphism (the Iitaka fibration of $X$) with connected fibres.
	Since the $f^*$ in the \eqref{eqinjhom} above is an isomorphism,
	$f\colon X\to X$ descends to an automorphism $f|_{B_m}$ on the base $B_m$
	of the Iitaka fibration $\Phi_{|sP|}\colon V\to B_m$.
	Note that $\dim B_m=\overline{\kappa}(X)=1$.

	Let $\pi=(\Phi_{|sP|})|_X\colon X\to B\coloneqq\pi(X)$.
	Then $f|_{B_m}$ restricts to (an automorphism) $f|_B$.
	Taking normalisation, we may assume that $B$ is smooth.
	Then $\overline{\kappa}(F)=0$ for a general fibre $F$ of $\pi$,
	by the definition of the Iitaka fibration (cf.~\cite{Iit82}).
	Now $F$ is a smooth affine curve with $\overline{\kappa}(F)=0$,
	so $F \cong \G_m$ i.e., $\pi$ is a $\G_m$-fibration.
	By \cref{l:euler}, $e(X)=0$.
	By the Suzuki formula (cf.~\cite{Suz77}, \cite{Gur97}), every fibre of $\pi$ has support $\G_m$.

	Let $\{m_iF_i\}_{i=1}^r$ ($r\geq 0$) be the set of multiple fibres $m_iF_i$ of $X\to B$
	lying over a point $b_i$.
	Set $B_0=B\setminus\{b_1,\dots,b_r\}$, $X_0=\pi^{-1}(B_0)$.
	Then $f|_{B_0}$ is induced by \cref{p:bad_fib}.
	By \cref{l:smooth_K*}, $\overline{\kappa}(B_0)=\overline{\kappa}(X_0)\geq\overline{\kappa}(X)=1$,
	so $\overline{\kappa}(B_0)=1$.
	Therefore $f|_{B_0}$ (and hence $f|_B$) are automorphisms of finite order, say the identity, after iterating $f$ (cf.~\cref{l:finord}).
	Moreover, either $B$ is irrational or $r\geq 3$.
	By \cref{p:fat_fib}, there is a finite surjective morphism $B'\to B$, from smooth $B'$, such that the normalisation $X'$ of $X \times_B B'$ is an {\'e}tale cover of $X$ with the projection $\pi'\colon X'\to B'$ being a $\G_m$-bundle.

	Now $f$ lifts to an endomorphism $f'$ on $X'$ such that $\pi'\circ f'=\pi'$.
	Applying \cite[Claim 3.2a]{GZ08}, there is a finite {\'e}tale cover $B''\to B'$ such that $\pi''\colon X''=X'\times_{B'} B''\to B''$ is a trivial $\G_m$-bundle.
	Clearly, $f'$ lifts to an endomorphism $f''$ on $X''$ such that $\pi''\circ f''=\pi''$.
\end{proof}

\begin{proposition}\label{p:torus_lift}
	Let $X$ be a smooth affine surface over $\kk$, and $f\colon X\to X$ a finite surjective morphism of degree $\geq 2$.
	Suppose $\overline{\kappa}(X)=0$.
	Then there is a finite {\'e}tale cover $\pi_T\colon T\to X$ from an algebraic torus $T\cong\G_m^2$ and $f$ lifts to an endomorphism $f_T$ on $T$.
\end{proposition}

\begin{proof}
	Assume first that $\kk=\C$.
	We begin with:

	\begin{claim}\label{claim:fib}
		$X$ is a \textit{Q}-algebraic torus (over $\kk=\C$).
	\end{claim}

	\begin{proof}[Proof of \cref{claim:fib}]
		By \cref{l:euler}, the topological Euler number $e(X)=0$.
		This and the Artin vanishing $H^j(X,\C)=0$ ($j>\dim X=2$) for affine varieties (cf.~\cite[Corollaire 3.5]{Art73}) imply the first Betti number $b_1(X)\geq 1$.
		Let $(V,D)$ be a log smooth compactification of $X$.
		The $E_1$-degeneration of the logarithmic Hodge-to-de Rham spectral sequence (cf.~\cite[Corollaire~3.2.13]{Del71}) implies:
		\begin{align*}
			1\leq b_1(X) & =h^0(V,\Omega_V^1(\log D))+h^1(V,\mathcal{O}_V)                     \\
			             & =h^0(V,\Omega_V^1(\log D))+h^0(V, \Omega_V^1)\leq 2\overline{q}(X),
		\end{align*}
		where the log irregularity $\overline{q}(X)\coloneqq h^0(V,\Omega_V^1(\log D))$.
		Hence the quasi-Albanese map $a\colon X\to S$ is non-trivial (to a semiabelian variety) with $\dim S=\overline{q}(X) \ge 1$.

		By \cite[Theorem 28]{Kaw81}, the map $a$ is dominant with general fibre irreducible.
		Let $T\subseteq S$ be the maximal subtorus and $A\coloneqq S/T$, a (projective) abelian variety.
		Then we get the dominant composition $b\colon X\to S\to A$.

		Suppose first that $\dim S=2=\dim A$.
		Then $S=A$ and $a\colon X\to A$ is a birational morphism.
		Take a log smooth compactification $X \subseteq V$ such that the map $a$ extends to a morphism $\overline{a}\colon V\to A$.
		Since the map $\overline{a}$ is also birational, we have $\kappa(V)=0$.
		But this contradicts $\overline{\kappa}(X)=0$ since $D=V\setminus X$ is a big divisor and so is $K_V+D$ (cf.~\cref{l:D_amp}).
		Therefore, this case cannot happen.

		When $\dim S=1$ (resp.~$\dim A=1$) we set $Y\coloneqq S$ (resp.~$Y\coloneqq A$);
		when $S=T$ and $\dim T=2$, we let $T\to Y\coloneqq\G_m$ be any projection;
		let $b: X\to Y$ be the natural dominant morphism, and $F$ its general fibre.
		By Iitaka's subadditivity (cf.~\cite[Theorem 11.15]{Iit82}),
		$0=\overline{\kappa}(X) \ge \overline{\kappa}(F)+\overline{\kappa}(b(X))\geq\overline{\kappa}(F)+\overline{\kappa}(Y)$.
		By the easy addition (cf.~\cite[Theorem 11.9]{Iit82}), $0=\overline{\kappa}(X) \le \overline{\kappa}(F)+\dim Y$; hence $\overline{\kappa}(F)\neq -\infty$.
		Combining the above all, we get $\overline{\kappa}(F)=0=\overline{\kappa}(Y)$ and $b$ is surjective.
		Hence $F\cong\G_m$ since $F\subseteq X$ is affine, and either $Y$ is an elliptic curve or $Y\cong\G_m$.
		Thus $X\to Y$ is a surjective $\G_m$-fibration.
		Since $e(X)=0$ by \cref{l:euler}, every fibre of $\pi$ has support $\G_m$ by the Suzuki formula (cf.~\cite{Suz77}, \cite{Gur97}).

		By \cref{p:fat_fib}, there is a finite surjective morphism $\phi\colon B'\to B$ such that the normalisation $X'$ (still affine) of $X \times_B B'$ is {\'e}tale over $X$ and the induced morphism $\pi'\colon X'\to B'$ is a $\G_m$-bundle.
		By \cref{l:smooth_K*}, $\overline{\kappa}(B')=\overline{\kappa}(X')$ ($= \overline{\kappa}(X)=0$).
		This and \cref{c:smooth_K*} imply that $B' \cong \G_m$.
		By \cref{l:smooth_K*} again, there is a finite {\'e}tale cover $\theta:B''\to B'$ (of degree $\le 2$) such that the base change $\pi''\colon X''\to B''$ of $\pi'$ via $\theta$ is a trivial $\G_m$-bundle.
		Since $B'$ is isomorphic to $\G_m$ so is its {\'e}tale cover $B''$.
		Now $X''$ is a trivial $\G_m$-bundle over $B'' \cong \G_m$, so $X'' \cong \G_m^2$.
		This proves \cref{claim:fib}.
	\end{proof}

	We return back to the proof of \cref{p:torus_lift}.
	Now for a general field $\kk$, we can use the same argument as in \cref{r:red_2C} to conclude that $X$ is a \textit{Q}-algebraic torus over $\kk$, with the help with \cref{claim:fib,lem:extsemiabel};
	note that any (connected) finite {\'e}tale cover of $X_{\kk}$ comes from a such one of $X_{\C}$ (cf.~\cref{lem:pi_1-inv}).
	Then, as a consequence of \cref{lem:torus-cls}, $f$ lifts to the algebraic torus closure $\pi_T\colon T\cong\G_m^2\to X$.
\end{proof}

Now we are ready for the following two proofs.

\begin{proof}[Proof of \cref{Thm_k>0_str}]
	It follows from Propositions~\ref{p:kappa=1_per} (and its proof) and~\ref{p:torus_lift}.
\end{proof}

\begin{proof}[Proof of \cref{Thm_k>0_conj}]
	By \cref{rem:AZO_to_ZDO}, for the ZDO~\ref{Conj:ZDO} part, we only need to prove AZO~\ref{Conj:AZO};
	further, for the AZO~\ref{Conj:AZO} part, we may assume that $(X,f)$ is defined over an algebraically closed field whose transcendence degree over $\Q$ is finite.
	For the KSC~\ref{Conj:KSC} part, as usual, we assume that $(X,f)$ is defined over $\overline{\Q}$.
	By \cite[Corollary 3.33]{Xie19}, we may assume that $\deg(f)\geq 2$.

	If $\overline{\kappa}(X)=2$, then $f$ is an automorphism of finite order by \cref{l:finord}, contradicting the
	extra assumption that $\deg(f)\geq 2$.

	If $\overline{\kappa}(X)=1$, then \cref{p:kappa=1_per} implies that there is a finite {\'e}tale cover $\phi\colon X''\to X$ such that $X''$ admits a $\G_m$-bundle structure $\pi''\colon X''\to B''$ and some iteration of $f$ lifts to $f''\colon X''\to X''$ and satisfies $\pi''\circ f''= \pi''$.
	Thus AZO~\ref{Conj:AZO} and KSC~\ref{Conj:KSC} hold for $(X'',f'')$ and hence also for $(X,f)$ (cf.~\cref{l:iter_inv}, \cite[Lemma~3.30]{Xie19} and \cref{p:kscequ}).

	If $\overline{\kappa}(X)=0$, then by \cref{p:torus_lift}, there is a finite {\'e}tale cover $\phi\colon\G_m^2\to X$ such that $f$ lifts to $f'\colon\G_m^2\to \G_m^2$.
	AZO~\ref{Conj:AZO} for semiabelian varieties is proved in \cref{Thm_semiAb} and KSC~\ref{Conj:KSC} for them is proved in \cite{MS20}.
	So AZO~\ref{Conj:AZO} and KSC~\ref{Conj:KSC} hold for $(X,f)$ too (cf.~\cite[Lemma 3.30]{Xie19} and \cref{p:kscequ}).
\end{proof}

\section[Extensions of polynomial maps:\\\mbox{} Proofs of Theorems~\ref{Thm_inf_pi1} and \ref{Thm_tri_poly}]{Extensions of polynomial maps: Proofs of\linebreak Theorems~\ref{Thm_inf_pi1} and \ref{Thm_tri_poly}}

The aim of this section is to show that Adelic Zariski Dense Orbit Conjecture (AZO~\ref{Conj:AZO})
is stable under extension by polynomial maps (cf.~\cref{thmskewprodpoly}).
With the help of this, we prove \cref{Thm_inf_pi1,Thm_tri_poly}.
The aim of this section is to prove \cref{Thm_semiAb} the proof of
which is similar to the proof of \cite[Theorem 1.14]{Xie19}.
See Subsection~\ref{setup:AdelicTopology} and \cref{rem:intersection_of_finite_adelic_open}
for a briefly introduction of adelic topology and a basic example of adelic open subset.
We follows the notations from there.
See \cite[Section 3]{Xie19} for the definition and more detailed discussions of the adelic topology.

Let $\kk$ be an algebraically closed field of finite transcendence degree over $\Q$.
We prove:

\begin{thm}\label{thmskewprodpoly}
	Let $X$ be a projective variety over $\kk$, and
	$f_1\colon X\dashrightarrow X$ a dominant rational self-map.
	Assume that AZO~\ref{Conj:AZO} holds for the pair $(X,f_1)$
	(this is the case when $\dim X=1$, cf.~\cref{l:iter_inv}).
	Let $f_2\in\kk(X)[y]\setminus\kk(X)$ be a nonconstant polynomial.
	Then AZO~\ref{Conj:AZO} and hence ZDO~\ref{Conj:ZDO}
	hold for the dominant rational self-map $f\colon X\times\PP^1\dashrightarrow X\times\PP^1$
	sending $(x,y)$ to $(f_1(x),f_2(x,y))$.
\end{thm}

\begin{definition}
	We say a pair $(X,f)$ satisfies
	\emph{Strong Adelic Zariski Dense Orbit}-property (SAZO-property for short)
	if the $f$-orbit of an adelic general point is well-defined and Zariski dense in $X$.
\end{definition}

\begin{proof}[Proof of \cref{thmskewprodpoly}]
	The proof is similar to the proof of \cite[Theorem~4.1]{Xie19}.
	We may replace $f_1$ and $f$ by iterations (cf.~\cref{l_inv_fun_field,l:iter_inv}).

	Denote by $\pi\colon X\times\PP^1\to X$ the first projection.
	If $H\in\kk(X)^{f_1}\setminus\kk$, then $\pi^*H\in\kk(X\times\PP^1)^f\setminus\kk$.
	So we may assume that SAZO-property holds for $(X,f_1)$.
	For every $z\in X\times \PP^1$ (resp.~$x\in X$), denote by $Z_f(z)$ (resp.~$Z_{f_1}(x)$) the Zariski closure of the $f$-orbit $O_f(z)$ of $z$ (resp.~the $f_1$-orbit $O_{f_1}(x_1)$ of~$x$).

	If $\deg_y(f_2)=1$, \cref{thmskewprodpoly} holds by \cite[Theorem~3.34]{Xie19}.
	Now assume $d\coloneqq\deg_y(f_2)\geq 2$.
	Write $f_2(x,y)=\sum_{i=0}^d a_i(x)y^i$ with $a_i\in\kk(X)$.
	There is a Zariski dense open subset $X'$ of $X$ such that $a_i\in\mathcal{O}(X')$ ($0 \le i \le d$) and $a_d(x)\neq 0$ for any $x \in X'$.

	We may assume that there is a nonempty adelic open subset $A \subseteq X$ such that for every $x\in A$, the $f_1$-orbit of $x$ is well-defined and Zariski dense in $X$.
	Let $K$ be a subfield of $\kk$ which is finitely generated over $\Q$, such that $\overline{K}=\kk$ and $f_1,f_2,X$ are defined over $K$.

	By \cite[Proposition 3.24]{Xie19}, replacing $f$ by an iteration,
	there is a nonempty adelic open subset $B\subseteq (X\times \PP^1)(\kk)$ such that for every point $z\in B$,
	the $f$-orbit of $z$ is well-defined and $Z_f(z)$ is irreducible.
	By \cite[Proposition 3.24]{Xie19} again, replacing $f$ by a positive power,
	we may assume that there exist a prime $p\geq 3$,
	an embedding $i\colon K\hookrightarrow\C_p$,
	and an open subset $V\simeq (\C_p^{\circ})^{\dim X}$ of $X_K(\C_p)$ which is $f_1$-invariant,
	such that the $f_1$-orbits of the points in $V$ are well-defined and $f_1|_V=\id\,\bmod\; p$.
	Moreover, there is an analytic action $\Phi\colon\C_p^{\circ}\times V\to V$ of $(\C_p^{\circ},+)$ on $V$ such that for every $n\in \Z_{\geq 0}$, we have $\Phi(n,\cdot)=f_1^{n}|_{V}(\cdot)$.
	In particular, $f_1^{p^n}(x)\to x$ when $n\to\infty$ for every $x\in V$.

	There is some $M\geq 1$ such that for every $i=0,1,\dots,d$ and $x\in V$,
	we have $\vert a_i(x)\vert\leq M$ and $\vert a_d(x)\vert\geq M^{-1}$.
	Pick some $R>M^2$.
	Let $U$ be the disc $\{\vert y\vert\geq R\}\cup\{\infty\}$ in $\PP^1(\C_p)$, where $y$ is the affine coordinate of $\PP^1$.
	Then $f$ is well-defined on $V\times U$ and $f(V\times U)\subseteq V\times U$.
	Moreover, for every $(x,y)\in V\times U$, $\ell\geq 0$,
	we have $f^{\ell+p^n}(x,y)\to (f_1^{\ell}(x),\infty)\subseteq V\times\{\infty\}$ when $n\to\infty$.
	In particular, we have
	$Z_{f_1}(\pi(z))\times\{\infty\}\subseteq Z_f(z)$.
	This property is purely algebraic,
	so for every $z\in C\coloneqq (X\times\PP^1)_K(i,U\times V)$,
	we have $Z_{f_1}(\pi(z))\times\{\infty\}\subseteq Z_f(z)$.

	By \cite[Proposition 3.18]{Xie19} or \cref{rem:intersection_of_finite_adelic_open},
	$\pi^{-1}(A)\cap B\cap C\cap (X\times \A^1)$
	is a nonempty adelic open subset of $(X\times \PP^1)(\kk)$.
	For every $z\in \pi^{-1}(A)\cap B\cap C\cap (X\times \A^1)$, we have
	\begin{enumerate}
		\item the orbits of $z$ and $\pi(z)$ are well-defined;
		\item $Z_{f_1}(\pi(z))=X$;
		\item $Z_{f}(z)$ is irreducible;
		\item $Z_{f_1}(\pi(z))\times\{\infty\}\subseteq Z_f(z)$; and
		\item $z\in Z_f(z)\setminus(X\times\{\infty\})$.
	\end{enumerate}
	It follows that $Z_f(z)=X\times\PP^1$.
	This proves \cref{thmskewprodpoly}.
\end{proof}

\Cref{thmskewprodpoly} above is the key in the following proof.

\begin{theorem}\label{ThmC}
	Let $X\coloneqq \A^1 \times \G_m$.
	Let $f\colon X\to X$ be a finite surjective endomorphism.
	Then AZO~\ref{Conj:AZO} and hence ZDO~\ref{Conj:ZDO} hold for $(X,f)$.
\end{theorem}

\begin{proof}
	Note that $\overline{\kappa}(\G_m)=0$.
	By \cref{l:descending}, $f$ descends along the natural projection $\pi\colon X\to B=\G_m$.
	Hence $f$ is of the form in \cref{thmskewprodpoly}.
	The result follows.
\end{proof}

Now we are ready for the following two proofs.

\begin{proof}[Proof of \cref{Thm_inf_pi1}]
	AZO~\ref{Conj:AZO} is known when $\deg(f)=1$ (cf.~\cite[Corollary 3.33]{Xie19}).
	So we always assume further $\deg(f)\geq 2$ for the AZO~\ref{Conj:AZO} part.

	If $\overline{\kappa}(X)\geq 0$, then (1) and (2) follow from \cref{Thm_k>0_conj}.
	Suppose that $\overline{\kappa}(X)=-\infty$.
	Then there is an $\A^1$-fibration $\pi\colon X\to B$
	since $X$ is affine and by the open surface theory (cf.~\cite[Ch.~3, Theorem 1.3.2]{Miy01}).
	If $\overline{\kappa}(B)=-\infty$,
	then $f$ descends to an automorphism $f|_B$ on $B$ of finite order by applying \cref{p:big_pi_1}.
	In this case,
	both AZO~\ref{Conj:AZO} and KSC~\ref{Conj:KSC} are vacuously true (cf.~\cref{l_inv_fun_field}).
	So we may assume that $\overline{\kappa}(B)\geq 0$.
	By \cref{l:descending}, $f$ descends to a surjective endomorphism $f|_B$ of $B$.
	If $\deg(f)=1$, then $f|_B$ is an automorphism of $B$.
	We see that $d_1(f)=1$ (cf.~\cite[Theorem 4]{Dan20}, arXiv version) and hence (3) holds.

	We may assume $\deg(f)\geq 2$ for both (1) and (2).
	We may assume also that $\overline{\kappa}(B)=0$ and $\pi\colon X\to B$ is an $\A^1$-bundle;
	otherwise, $g$ is of finite order by \cref{lem:A1_fib} and \cref{p:bad_fib},
	and we are done as before.
	By \cref{lem:ell_base,lem:embd_in_P1xP1},
	AZO~\ref{Conj:AZO} and KSC~\ref{Conj:KSC} are true
	unless $X\cong\A^1\times\G_m$ (cf.~\cref{l_inv_fun_field}).
	Then (2) follows, and (1) holds true by \cref{ThmC}.
\end{proof}

\begin{proof}[Proof of \cref{Thm_tri_poly}]
	It follows from \cref{thmskewprodpoly} by induction on the factors.
\end{proof}

\section[Maps with larger first dynamical degree:\\\mbox{} Proof of Theorem~\ref{Thm_big_d1}]{Maps with larger first dynamical degree: Proof of Theorem~\ref{Thm_big_d1}}

In this section, we consider Zariski Dense Orbit Conjecture (ZDO~\ref{Conj:ZDO}) via the arithmetic degree.
Let $X$ be a projective variety over $\overline{\Q}$ and $f\colon X\to X$ a surjective morphism.

The following is a generalisation of \cite[Lemma 9.1]{MSS18} to the singular case, but the proof of \cite[Lemma 9.1]{MSS18} is valid even in the singular case.

\begin{proposition}\label{proextarhitd}
	Assume that $d_1(f)>1$.
	Let $D\not\equiv 0$ be a nef $\R$-Cartier divisor on $X$ such that $f^*D\equiv d_1(f)D$.
	Let $V\subseteq X$ be a subvariety of positive dimension such that $(D^{\dim V}\cdot V)>0$.
	Then there exists a nonempty open subset $U\subseteq V$ and a set $S\subseteq U(\overline{\Q})$ of bounded height such that for every $x\in U(\overline{\Q})\setminus S$ we have $\alpha_f(x)=d_1(f)$.
\end{proposition}

\begin{remark}\label{rempickcd}
	Let $C$ be an irreducible curve
	which is a complete intersection of $\dim X-1$ of ample effective divisors on $X$.
	Then $(D\cdot C)>0$.
\end{remark}

As defined in \cite{MMSZZ20},
an $f$-periodic subvariety $V$ is said to be of \emph{Small Dynamical Degree} (SDD for short)
if the first dynamical degree $d_1(f^s|_V)<d_1(f^s)$ for some $s\geq 1$ such that $f^s(V)=V$.

\begin{definition}
	We say that $(X,f)$ satisfies the \emph{SDD condition}
	if there is an $f^{-1}$-invariant Zariski closed proper subset $Z$ of $X$
	such that all irreducible $f$-periodic proper subvarieties not being contained in $Z$, are SDD\@.
\end{definition}

The SDD condition is a dynamical property of the
algebraic dynamical system \((X,f)\).

\begin{thm}\label{thmssdzdo}
	If $(X,f)$ satisfies the SDD condition, then ZDO~\ref{Conj:ZDO} holds for $(X,f)$.
\end{thm}

\begin{proof}
	Set $\ell\coloneqq\dim X$.
	By the assumption,
	there is an $f^{-1}$-invariant Zariski closed subset $Z$ of $X$
	such that all $f$-periodic proper subvarieties not being contained in $Z$, are SDD\@.

	Assume first $d_1(f)=1$.
	Since $(X,f)$ satisfies the SDD condition, there is no proper $f$-periodic subvarieties outside $Z$.
	Pick any point $x\in X(\overline{\Q})\setminus Z$
	and let $Z_f(x)$ be the Zariski closure of the $f$-orbit $O_f(x)$ of $x$.
	Then, for some $t\geq 1$,
	$f^t(x)$ is contained in a $f$-periodic subvariety of $Z_f(x)$ (cf.~\cite[Lemma 2.7]{MMSZ20}),
	which is hence either equal to $X$ or contained in $Z$.
	In the latter case, $x$ is contained in $f^{-t}(Z)=Z$, which is a contradiction.
	Thus the theorem is true when $d_1(f) = 1$.

	Now we may assume that $d_1(f)>1$.
	By the generalised Perron-Frobenius theorem due to Birkhoff,
	there is a nonzero nef $\R$-divisor $D\in\N^1(X)\coloneqq\NS(X) \otimes_{\Z} \R$ such that $f^*D\equiv d_1(f)D$.
	Let $H_1,\dots,H_{\ell-1}$ be general very ample divisors on $X$.
	By applying Bertini's theorem to the pullback of $|H_i|$ to a smooth model of $X$, we may assume that $C\coloneqq H_1\cap \dots \cap H_{\ell-1}$ is irreducible and it is not contained in $Z$.
	By \cref{rempickcd}, we may apply \cref{proextarhitd} to the curve $C$.
	By the Northcott property, there is a point $x\in C(\overline{\Q})\setminus Z$ with $\alpha_f(x)=d_1(f)$.

	For some $t\geq 1, s\geq 1$, our $f^t(x)$ is contained in an $f^s$-invariant irreducible component $V$ of $Z_f(x)$ (cf.~\cite[Lemma 2.7]{MMSZ20}).
	If $V=X$ then $Z_f(x)=X$ and we are done.
	If $V\subseteq Z$ then $x\in f^{-t}(V)\subseteq f^{-t}(Z)=Z$, absurd.
	Thus $V$ is not contained in $Z$, and is an $f^s$-invariant proper subvariety of $X$.
	Hence $V$ is SDD\@.
	So $d_1(f^s|_V)<d_1(f^s)$.
	Set $y\coloneqq f^t(x)\in V$.
	Then $\alpha_f(y)=\alpha_f(x)=d_1(f)$.
	Now (cf.~\cref{lem_subvar})
	\[
		d_1(f^s|_V)<d_1(f)^s=\alpha_f(y)^s=\alpha_{f^s}(y)=\alpha_{f^s|_V}(y),
	\]
	which contradicts \cref{proupboundarth}.
	This proves \cref{thmssdzdo}.
\end{proof}

It is clear that the SDD condition is satisfied when $\dim X=1$ and $\ord(f)=\infty$.
There are still some nontrivial examples where the SDD condition is satisfied.

\begin{proposition}\label{prodim23ssd}
	Let $f\colon X\to X$ be a dominant endomorphism of projective variety.
	Assume either one of the following two conditions.
	\begin{enumerate}
		\item $\dim X=2$, $d_1(f)>1$ and $d_1(f)\geq d_2(f)$; or
		\item $\dim X=3$, and $d_1(f)>d_3(f)=1$.
	\end{enumerate}
	Then either there is an $f^*$-invariant nonconstant rational function on $X$, or $(X,f)$ satisfies the SDD condition.
\end{proposition}

\begin{proof}
	Fix an embedding $\overline{\Q}\hookrightarrow\C$.
	If there are infinitely many $f^{-1}$-periodic pairwise component non-overlapping hypersurfaces of $X$,
	then we may find sufficiently many $f^{-1}$-invariant (not necessarily irreducible) hypersurfaces and hence so does for $(X_{\C},f_{\C})$.
	By \cite[Theorem B]{Can10}, $(f_{\C})^*$ preserves a nonconstant rational function on $X_{\C}$ and hence so does $(X,f)$ (cf.~\cref{l_inv_fun_field}).
	Thus we may assume that there is an $f^{-1}$-invariant hypersurface $Z$ such that for every hypersurface $H$ of $X$, if it is $f^{-m}$-invariant for some $m\geq 1$, then we have $H\subseteq Z$.

	Let $V$ be any irreducible $f$-periodic subvariety of period $m\geq 1$.
	If $\dim V=0$, then we have $d_1(f^m)>1=d_1(f^m|_V)$.

	Assume that $(1)$ holds.
	We may assume that $\dim V=1$ and $V\not\subseteq Z$.
	Hence $f^{-m}(V)=V\cup V'$ for some (nonempty) curve $V'\neq V$.
	Then $\deg(f^m|_V)<\deg(f^m)\leq d_1(f^m)$.

	Now assume that $(2)$ holds.
	Since $f$ is an automorphism, all $f$-periodic hypersurfaces are contained in $Z$.
	So we may assume that $\dim V=1$ and $V\not\subseteq Z$.
	Since $f$ is an automorphism and $V$ is a curve, $d_1(f^m|_V)=\deg(f^m|_V)=1<d_1(f^m)$.
\end{proof}

\begin{remark}
	The proof of \cref{prodim23ssd} still works when \(f\) is a rational self-map. We omit the details and leave its verification to interested readers
	(also because \cref{thmssdzdo} is not extendable to the rational map case at the moment).
\end{remark}

\begin{proposition}\label{prodimhigdim}
	Let $f\colon X\to X$ be a dominant endomorphism of a smooth projective variety of dimension $d\geq 2$.
	Suppose $d_1(f)>\max_{i=2}^{d}\{d_{i}(f)\}$.
	Then $(X,f)$ satisfies the SDD condition.
\end{proposition}

\begin{proof}
	Let $V$ be any irreducible $f$-periodic subvariety of period $m\geq 1$.
	We only need to show that $d_1(f^m|_V)<d_1(f^m)$.
	We may assume that $\dim V\geq 1$.
	After replacing $f$ by $f^m$, we may assume that $m=1$.

	Set $\ell\coloneqq\dim V\in\{1,\dots,d-1\}$.
	Denote by $q\colon\N^1(X)\to\N^1(V)$ the restriction homomorphism.
	Let $H$ be an ample class in $\N^1(X)$.
	Note that $q(\N^1(X))$ is an $(f|_V)^*$-invariant subspace of $\N^1(V)$.

	Suppose the contrary that $d_1(f|_V)=d_1(f)$.
	Then
	\[
		\lim_{n\to \infty}\Vert(f^n|_V)^*(q(H))\Vert^{1/n}=d_1(f|_V)=d_1(f)
	\]
	where $\Vert\cdot\Vert$ is any norm on $q(\N^1(X))$.
	Since the cone $\Nef(X)|_V$ contains an ample divisor on $V$,
	there is some $D\in\N^1(X)$ such that $D|_V\in\Nef(V)\setminus\{0\}$ and $(f|_V)^*(D|_V)\equiv d_1(f) D|_V$,
	noting that $d_1(f|_V)=d_1(f)$ by the extra assumption.
	In other words, we have $0\neq V\cdot D\in\N^{d-\ell+1}(X)$ (the real vector space of codimension-$(d-\ell+1)$ cycle classes modulo numerical equivalence), and $f^*D\equiv d_1(f) D+F$ where $F\cdot V\equiv 0$.
	Since $V$ is $f$-invariant, $f_*V=\deg(f|_V)V$.
	Since $f_*f^*= \deg(f)\id$ on $\N^{d-\ell}(X)$
	we get $f^*V\equiv bV$ where $b = \deg(f)/\deg(f|_V) \ge 1$.
	Then we have
	\[
		f^*(D\cdot V)=f^*D\cdot f^*V\equiv bd_1(f)(D\cdot V).
	\]
	Since $bd_1(f)\geq d_1(f)>d_{d-\ell+1}(f)$, we get a contradiction.
\end{proof}

Now we are ready for:

\begin{proof}[Proof of \cref{Thm_big_d1}]
	It follows from \cref{thmssdzdo} and \cref{prodim23ssd,prodimhigdim}.
\end{proof}

\begin{remark}
	The case (1) of \cref{Thm_big_d1} is already proved in \cite{JXZ20}.
	However, the method there is different.
\end{remark}

The following example shows that the SDD condition does not hold in general.

\begin{example}
	Let $X=\PP^2$ and $f\colon X\to X$ the endomorphism $(x,y)\to (x^2,y^2)$, in affine coordinates.
	Then we have $d_1(f)=2$.
	For two coprime positive integers $a,b$, the curve $C_{a,b}=\{x^ay^b=1\}$ is $f$-invariant and $d_1(f|_{C_{a,b}})=2$.
	This implies that SDD condition does not hold for the pair $(X,f)$.
\end{example}

\end{document}